\documentclass{amsart}

\usepackage{amsmath}
\usepackage{amsthm}
\usepackage{amssymb}
\usepackage[all]{xy}

\newtheorem{theorem}{Theorem}[section]

\newtheorem{proposition}[theorem]{Proposition}
\newtheorem{corollary}[theorem]{Corollary}

\theoremstyle{definition}
\newtheorem{definition}[theorem]{Definition}

\newcommand{\func}[1]{{\mathrm{#1}}}
\newcommand{\C}{\mathbf{C}}
\newcommand{\B}{\mathbf{B}}
\newcommand{\Set}{\ensuremath{\mathsf{Set}}}
\newcommand{\Cat}{\ensuremath{\mathsf{Cat}}}
\newcommand{\Mon}{\ensuremath{\mathsf{Mon}}}
\newcommand{\Grp}{\ensuremath{\mathsf{Grp}}}
\newcommand{\Ab}{\ensuremath{\mathsf{Ab}}}

\DeclareMathOperator{\cod}{cod}
\DeclareMathOperator{\dom}{dom}
\DeclareMathOperator{\map}{map}
\newcommand{\squarecell}[9]{
\xymatrix@R=1pc@C=5pc{
#1 \ar@{->}[r]|{#2}="2" \ar[d]_{#4} & #3 \ar[d]^{#6} \\
#7 \ar@{->}[r]|{#8}="8" & #9  \\
\ar@{=>}^{#5} "2";"8"}}
\newcommand{\quadrado}[8]{
\xymatrix{
#1 \ar[r]^{#2} \ar[d]_{#4}  & #3  \ar[d]^{#5} \\
#6  \ar[r]^{#7}  &  #8 }}
\newcommand{\cela}[5]{
\xymatrix{
#1  \ar@/^1pc/[r]^{#2}="2"  \ar@/_1pc/[r]_{#4 }="4"  &  #5  \\
\ar@{=>}|{#3 } "2";"4"}}
\newcommand{\rg}[5]{    \xymatrix{#1 \ar@<1ex>[r]^{#2} \ar@<-1ex>[r]_{#4} & #5  \ar[l]|{#3} } }
\newcommand{\splitepi}[4]{    \xymatrix{#1 \ar@<.5ex>[r]^{#2} & #4  \ar@<.5ex>[l]^{#3} } }
\newcommand{\splitext}[6]{    \xymatrix{#1 \ar[r]^{#2} & #3 \ar@<.5ex>[r]^{#4} & #6  \ar@<.5ex>[l]^{#5} } }

\begin{document}
% Title Block
\title[Pseudocategory internal to a sesquicategory]{On the notion of pseudocategory internal to a category with a 2-cell structure}
\author[N. Martins-Ferreira]{Nelson Martins-Ferreira}
\address{School of Technology and Management-ESTG\\
Centre for Rapid and Sustainable Product Development-CDRSP\\
Polytechnic Institute of Leiria\\
P-2411-901, Leiria, Portugal
%Departamento de Matem\'atica, Escola Superior de Tecnologia e
%Gest\~ao, \\  and Centre for Rapid and Sustainable Product Development, Instituto Poli\-t\'ecnico de Leiria, Portugal
}
\email{martins.ferreira@ipleiria.pt}

\thanks{This is the result of several talks given by the author in Milan
(Oct2006), Coimbra (May2007) and Haute-Bodeux (Jun2007); it is a revised version of \cite{NMF09arxiv}. Thanks are due to G. Janelidze, R. Street and T. Van der Linden.}

\begin{abstract}
The notion of pseudocategory, as considered in \cite{Ch1}, is extended from the context of a 2-category to the more general one of a sesquicategory, which is considered as a category equipped with a 2-cell structure. Some particular examples of 2-cells arising form internal transformations in internal categories, conjugations in groups, derivations in crossed-modules or homotopies in abelian chain complexes are studied in this context, namely their behaviour as abstract 2-cells in a 2-cell structure. Issues such as naturality of a 2-cell structure are investigated. This article is intended as a preliminary starting work towards the study of the geometrical aspects of the 2-cell structures  from an algebraic point of view.
\end{abstract}

\maketitle

% Text of Document.
\section{Introduction}

In this article we are using a different notation for the vertical composition of
2-cells: instead of the usual dot `$\cdot $' or bullet `$\bullet$' we use plus `$+$'. To support
this change of notation we present an  analogy which compares the geometrical vectors in the
plane with the 2-cells between morphisms in a category.%
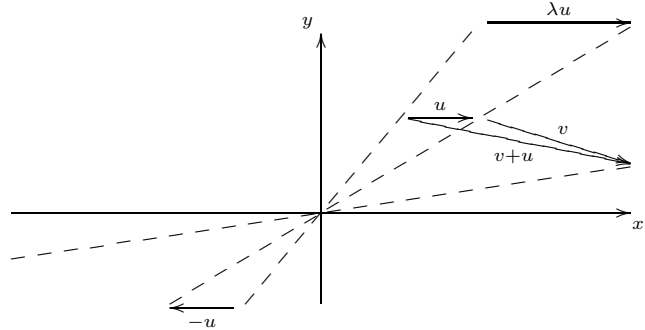
\begin{equation*}
\xymatrix@R=1pc@C=2pc{
&&&&&&
\ar@{->}[rr]^{\lambda u}
&& \\
&&&&&&&&\\
&&&&&
\ar@{->}[r]^{u}
\ar@{->}[rrrd]_{v+u}
&
\ar@{->}[rrd]^{v}
&&\\
&&&&&&&&\\
\ar[rrrrrrrr]_(1){x}
&&&&&&&&\\
\ar@{--}[rrrrrrrruu]
&&&&&&&&\\
&&
\ar@{--}[rrrrrruuuuuu]
&
\ar@{--}[rrruuuuuu]
\ar@{->}[l]^{-u}
&
\ar[uuuuuu]^(1){y}
&&&&\\}
\end{equation*}%
Two geometrical vectors on the plane can be added only if the end point of
the second one ($u$ as in the picture above) is the starting point of the first
one ($v$ as in the picture), in that case the resulting vector, which is their sum,
goes from the starting point of the second one to the end point of the first one. The same happens
for 2-cells in a category:%
\begin{equation*}
\xymatrix{
\cdot
  \ar@/^2pc/[r]^{\func{dom}u}="1"
  \ar[r]|{}="2"
  \ar@/_2pc/[r]_{ \cod v}="3"
 &   \cdot
 \\
\ar@{=>}^{u } "1";"2"
\ar@{=>}^{v } "2";"3"}
\longmapsto
\cela {\cdot } {\dom u} {v+u} {\cod v} {\cdot};
\end{equation*}%
In some sense the analogy can be extended for scalar multiplication%
\begin{equation*}
\xymatrix{
\cdot
\ar[r]^{\lambda}
& \cdot
  \ar@/^1pc/[r]^{\dom u}="1"
  \ar@/_1pc/[r]_{\cod u }="2"
 &   \cdot
  \ar[r]^{\rho}
 & \cdot \\
\ar@{=>}^{u } "1";"2"}\longmapsto \cela {\cdot } {\rho\dom(u)\lambda } {\rho u \lambda} {\rho\cod(u)\lambda} {\cdot},
\end{equation*}%
and for inverses (in the case they exist, and in a much lesser degree of analogy)%
\begin{equation*}
\cela {\cdot } {\dom u} {\, u} {\cod u} {\cdot} \longmapsto
\cela {\cdot } {\cod u} {-u} {\dom u} {\cdot}.
\end{equation*}%
Concerning horizontal composition, there is still an analogy with some
relevance: it is, in some sense, analogous to the cross product of vectors --- in the sense that it raises in dimension (see the introduction of \cite{Crans2} and its references for further discussion on this). Given 2-cells, $u$ and $v$, displayed as,
\begin{equation*}
\xymatrix{
\cdot
  \ar@/^1pc/[r]^{\dom u}="1"
  \ar@/_1pc/[r]_{\cod u }="2"
 &  \cdot
  \ar@/^1pc/[r]^{\dom v}="3"
  \ar@/_1pc/[r]_{\cod v }="4"
 & \cdot \\
\ar@{=>}^{ u } "1";"2"
\ar@{=>}^{ v } "3";"4"}
\end{equation*}%
the horizontal composition, usually denoted by $v\circ u$, is expected to be an isomorphic 3-cell, from the 2-cell%
\begin{equation}
\cod \left( v\right) u+v\dom \left( u\right)  \label{f1}
\end{equation}%
to the 2-cell%
\begin{equation}
v\cod \left( u\right) +\dom \left( v\right) u.  \label{f2}
\end{equation}%
In some cases $\left( \ref{f1}\right) $ and $\left( \ref{f2}\right) $
coincide (as it happens in a 2-category) and this is the reason why we may consider  a horizontal composition. 

The purpose of this work is to extend the notion of pseudocategory internal to a 2-category \cite{Ch1} to the more general context of a pseudocategory internal to a category with a 2-cell structure (or sesquicategory) in which case $(\ref{f1})$ and $(\ref{f2})$ are not necessarily equal.

To do this we have to consider the horizontal composition as a relation, written as $v\circ u$, saying that the 2-cell $v$ is
natural with respect to the 2-cell $u$, which is defined by%
\begin{equation*}
v\circ u\Longleftrightarrow \left( \ref{f1}\right) =\left( \ref{f2}\right).
\end{equation*}%

With this regard, the horizontal composition is only defined for those pairs $\left( v,u\right)$ that are in relation $v\circ u$, with the composite
being given, in that case, by either $\left( \ref{f1}\right) $ or $\left( \ref{f2}%
\right)$.

This is a geometrical intuition. An algebraic intuition is provided in
Proposition \ref{giving a 2-cell structure}. It is also possible to put the geometrical analogy into a formal context. It is obtained by taking the category whose objects are the straight lines passing through the origin while the morphisms are the points in the plane except the origin. The domain and codomain associate to each non-null vector its direction and the composition of two vectors with the same direction is the product of their signed magnitudes as real numbers. Then the geometrical vectors in the plane appear in this situation as the codiscrete 2-cell structure.

This article is organized as follows.
In section \ref{section2} we recall the notions of internal (pre)category, internal (pre)functor and internal (natural) transformation. This section  is also used to introduce some notation. In section \ref{section 3} we consider an arbitrary fixed category, $\mathbf{C}$, and define a 2-cell structure over it, as to make it a sesquicategory. We give a characterization of that structure as a family of sets, together with maps and actions,
satisfying some conditions (Proposition \ref{giving a 2-cell structure}). This characterization generalizes the description of
2-Ab-categories as families of abelian groups, together with group
homomorphisms and laws of composition, which may be found in \cite{Ch5} and \cite{Ch6}, except that now, the (strong) condition%
\begin{equation}\label{cond Dxy=xDy}
D\left( x\right) y=xD\left( y\right)
\end{equation}
is no longer a requirement. A useful consequence of this is  the possibility of considering the example of chain
complexes, say of order 2, which is treated in section \ref{ex13}. The condition $(\ref{cond Dxy=xDy})$ above is equivalent to the naturality condition, and since the
results obtained in \cite{Ch5} and \cite{Ch6} rely on this
assumption, we have to be careful in removing it. With that regard we introduce and
study the concept of a 2-cell being natural with respect to another 2-cell (Definition \ref{def nat2cellwrt}),
and the concept of a natural 2-cell (Definition \ref{def natcell}), which is natural with respect to all the possible 2-cells whose horizontal composition with it makes sense (section \ref{section 5}). There is also an entire section dedicated to a list of examples (section \ref{section 4}).

During the last part of the paper (sections \ref{section 5} and \ref{section 6}) we work towards the definition of pseudocategory internal to an arbitrary category equipped with a 2-cell structure. For that we need the notion of cartesian square with respect to a specified 2-cell strucutre (Definiton \ref{Def: cartesian 2-cell structure, Ch2}), and the notions of natural and invertible 2-cell structures. We compare the abstract notions of naturality, which are being introduced, when $\mathbf{C}$ is a category of the form $\Cat\left( \mathbf{B}\right)$, of internal categories in some category \textbf{B}. As expected, when the 2-cell structure consists on the internal
transformations (not necessarily natural) then every natural transformation  is a natural 2-cell (Corollary \ref{corolary1}).
%Moreover in order to check if a certain 2-cell is natural it is sufficient to check if it is
%natural with respect to a particular 2-cell (from
%the \emph{ category of arrows}).
%
%We give a general process for constructing 2-cell structures in arbitrary
%categories, and for the purposes of latter discussions we will restrict our
%study to the 2-cell structures obtained this way. In order to argue that we
%are not restricting too much, we show that the canonical 2-cell structures
%over groups and crossed-modules, that are respectively \textquotedblleft
%conjugations" and \textquotedblleft derivations", are captured by this
%construction.

%The notion of cartesian square is used for instance  to consider
%2-cells of the form $u\times _{w}v$ that are used in the coherence
%conditions involved in a pseudocategory.

At the end we extend the notion of pseudocategory from the context of a
2-category to the more general context of a category with a 2-cell structure
(sesquicategory). As an example of application we consider the
sesquicategory of abelian chain complexes with homotopies as 2-cells and
study pseudocategories in there. We also give some results extracted from \cite{Ch9} concerning (weakly) Mal'tsev sesquicategories.

All the notions defined in \cite{Ch1}: pseudofunctor, natural and
pseudo-natural transformation and modification, can be extended in this
way. %
However, these considerations are to be developed in a future work. This paper
is the starting point for a systematic study of internal categorical
structures in a category with a given 2-cell structure, and also to
investigate how these categorical structures are changed when the given
2-cell structure over the (fixed) base category varies. For example, a
pseudocategory, in a category with the discrete 2-cell structure, is an
internal category, while if the 2-cell structure is the codiscrete one, then it
is just a precategory.

This work is a revised version of \cite{NMF09arxiv}.

\section{Internal precategories}\label{section2}

A usual assumption on a category $\C$, which is frequently asked when working with categories internal to $\C$, is the existence of pullbacks. Indeed, a minimal
requirement is the existence of pullbacks of split epimorphisms along split
epimorphisms. In this work we are interested in the notion of internal
(pre)category without assuming the existence of pullbacks on the ambient category $\C$. This means that we have to
consider an internal (pre)category as a structure where some conditions are satisfied including, in particular, the requirement that some squares have  the property of being pullback squares. This approach is useful, for example,
in the study of (pre)categories internal to arbitrary ambient categories. Later on, in this text, when considering the notion of pseudocategory internal to a category with a 2-cell structure, we will also be considering pullback squares that share their universal property with morphisms and 2-cells, they will be called cartesian squares. There are currently several slightly different notions of precategory in the literature, mainly with respect to what concerns axiom (PC3) below (with the main reference on this topic being \cite{Janelidze01}). Here we are freely using the notion which best fits the setting of our work. This notions will be considered first in section \ref{ex7}, and then in Proposition \ref{prop_naturality of transformations}.

Let $\C$ be an arbitrary category. A precategory internal to $\C$ is a diagram of the shape \[\xymatrix{A_2\ar[r]|{m}\ar@<1ex>[r]^{\pi_2}\ar@<-1ex>[r]_{\pi_1}&A_1\ar@<1ex>[r]^{d}\ar@<-1ex>[r]_{c}&A_0\ar[l]|{e}}\] satisfying the following conditions:
\begin{enumerate}
\item[PC1] $de=1_{A_0}=ce$,
\item[PC2]  $dm=d\pi_2$ and $cm=c\pi_1$,
\item[PC3] the square \[\xymatrix{A_2\ar[r]^{\pi_2}\ar[d]_{\pi_1}&A_1\ar[d]^{c}\\A_1\ar[r]^{d}&A_0}\] is a pullback square.
\end{enumerate}

%In practical terms, condition (PC3) is saying that the pair of arrows $(d,c)$ has a pullback square which coincided with the domain of the morphism $m$, hence we will always refer to the structure of a precategory as a five-tuple $(A_0,A_1,d,c,e,m)$.

By condition (PC3) we mean that the square is commutative and that the usual universal property of pullbacks is satisfied.

A precategory $A$, internal to $\C$, will be represented as a nine-tuple \[(A_0,A_1,A_2,d,c,e,m,\pi_1,\pi_2).\] 
In some cases thought, and in order to simplify notation,  when all pullbacks exist in $\C$ or when it is not important to specify the pullback $A_2$ with its projections, we will refer to a precategory $A$ simply as a six-tuple \[A=(A_0,A_1,d,c,e,m).\]

If $A=(A_0,A_1,d,c,e,m)$ and $B=(B_0,B_1,d',c',e',m')$ are two internal precategories then a morphism between them is called a prefunctor and it consists in a pair $(f_0,f_1)$ of morphisms, $f_i\colon{A_i\to B_i}$, $i=0,1$, such that $d'f_1=f_0d$ and $c'f_1=f_0c$.

A transformation between two internal prefunctors, say from $f=(f_0,f_1)$ to $g=(g_0,f_1)$, both from a precategory \[A=(A_0,A_1,A_2,d,c,e,m,\pi_1,\pi_2)\] to a precategory \[B=(B_0,B_1,B_2,d',c',e',m',\pi'_1,\pi'_2),\] is a triple\[t=(t_1,t_2,t_3)\] with $t_1\colon{A_0\to B_1}$, $t_2,t_3\colon{A_1\to B_2}$ morphisms in $\C$ such that 
\begin{eqnarray}
dt_1=f_0 &,& ct_1=g_0\nonumber\\
\pi_1t_2=t_1c &,& \pi_2t_2=f_1\nonumber\\
\pi_1t_3=g_1 &,& \pi_2t_3=t_1d.\nonumber
\end{eqnarray} 
Note that $B_2$ is a pullback, hence $t_2=\langle t_1c,f_1\rangle$ and $t_3=\langle g_1,t_1d \rangle$.

An internal category is a precategory $( C_{0},C_{1},C_2,d,c,e,m,\pi_1,\pi_2)$ in which the following conditions are satisfied:
\begin{enumerate}
\item[PC4] $m\langle ec,1_{C_1}\rangle=1_{C_1}=m\langle 1_{C_1}, ed\rangle$,
\item[PC5] there exists a span \[\xymatrix{C_2&C_3\ar[l]_{p_1}\ar[r]^{p_2}&C_2}\] such that the square \[\xymatrix{C_3\ar[r]^{p_2}\ar[d]_{p_1}&C_2\ar[d]^{\pi_1}\\C_2\ar[r]^{\pi_2}&C_1}\] is a pullback square,
\item[PC6] $m(1_{C_1}\times_{C_0} m)=m(m\times_{C_0} 1_{C_1})$.
\end{enumerate}

An internal functor is a prefunctor $f\colon{A\to B}$ of internal precategories (that are internal categories) which, in addition to $d'f_1=f_0d$ and $c'f_1=f_0c$, satisfies $f_1e=e'f_0$ and $f_1m=m'f_2$, with $f_2=f_1\times_{f_0}f_1$.

A transformation $t=(t_1,t_2,t_3)\colon{f\to g}$ is a natural transformation if, in addition to the conditions above, $m't_2=m't_3$.

The purpose of this work is to consider a setting which is appropriate to the handling of internal categories and internal precategories as two extreme cases of the more  general notion of pseudocategory. As it will be clear from the next section, every category can be equipped with several 2-cell structures. Each one of them will give a different notion of a pseudocategory. In particular it is always possible to define two trivial 2-cell structures on the same ambient category: the discrete and the codiscrete ones. An internal category is a pseudocategory with respect to the discrete 2-cell structure while a precategory is a pseudocategory with respect to the codiscrete 2-cell structure. Our interest will be focused on the notions which may arise as intermediate cases.

\section{Categories with 2-cell structures or sesquicategories}\label{section 3}

As already motivated by the introduction, this work considers the notion of a category with a 2-cell structure, which is the same thing as a sesquicategory. We prefer to call it a category with a 2-cell structure in order to emphasise the possibility of specifying  different 2-cell structures over the same ambient base category. The example which motivates this approach is the category of (left and right) $R$-modules for some unitary ring $R$ which also suggests the additive notation that is being used throughout the text.

Let $\mathbf{C}$ be an arbitrary but fixed category. Its \emph{hom} functor \[\hom_{\C}\colon{\C^{op}\times \C\to\Set}\] will be simply referred to as $\hom$.

A 2-cell structure on $\C$, or over $\C$, is a category object, $(C_0,C_1,d,c,e,m)$, internal to the functor category $\Set^{\C^{op}\times \C}$ whose object of objects, $C_0$, is the functor $\hom$. 

\begin{definition}[2-cell structure]\label{def 2cellstruct}
A 2-cell structure on $\mathbf{C}$ is a system%
\begin{equation*}
\mathbf{H}=\left( H,\dom ,\cod ,0,+\right)
\end{equation*}%
with
\begin{equation*}
H\colon{\C^{op}\times \C\to\Set}
\end{equation*}%
\newline
a functor, $\dom,\cod,0,+$ natural transformations, displayed as,% 
\begin{equation*}
H\times _{\hom }H\overset{+}{\longrightarrow }H%
\begin{array}{c}
\underrightarrow{\dom } \\
\overset{0}{\longleftarrow } \\
\overrightarrow{\cod }%
\end{array}%
\hom 
\end{equation*}%
and such that the six-tuple%
\begin{equation*}
\left( \hom ,H,\dom ,\cod ,0,+\right)
\end{equation*}%
is an internal category in the functor category $\Set^{\C^{op}\times \C}$.
\end{definition}

A category $\C$, equipped with a 2-cell structure $\mathbf{H}$, is represented as a pair $(\C,\mathbf{H})$, and it is a sesquicategory (see for example \cite{Street,Street2,str3}).

\begin{proposition}Every category with a 2-cell structure over it is a sesquicategory. Every sesqui\-category determines a 2-cell structure over its underlying category.
\end{proposition}
\begin{proof}
 A sesquicategory is a category, $\mathbf{C}$, with a functor, $\mathbf{L}\colon{\C^{op}\times \C\to\Cat}$, into $\Cat$,
such that its restriction to $\Set$ (by forgetting the arrows, $\pi_0$) gives $\hom _{\mathbf{C}}$ and its restriction to $\Set$ (by forgetting the objects, $\pi_1$) gives $\mathbf{H}$, as illustrated.
\begin{equation*}
\xymatrix{ & \Cat \ar@<-.5ex>[d]_{\pi_1}\ar@<.5ex>[d]^{\pi_0} \\ \mathbf{C}^{op} \times \mathbf{C}
\ar[ur]^{\mathbf{L}} \ar@<-.5ex>[r]_{\hom}\ar@<.5ex>[r]^{\mathbf{H}} & \Set }
\end{equation*}
\end{proof}

The notation introduced in the following proposition will be used throughout the text. It gives a detailed description of the whole information which is needed to equip a category $\C$ with a 2-cell structure $\mathbf{H}=(H,\dom,\cod,0,+)$. This notation is borrowed from the motivating example of left and right modules over a ring.

\begin{proposition}
\label{giving a 2-cell structure}Giving a 2-cell structure over a category $%
\mathbf{C}$, is to give, for every pair $\left( A,B\right) $ of objects in $%
\mathbf{C}$, a set $H\left( A,B\right) $, together with maps%
\begin{equation}\label{intcat Hom(A,B)}
H\left( A,B\right) \times _{\hom \left( A,B\right) }H\left( A,B\right)
\overset{+}{\longrightarrow }H\left( A,B\right)
\begin{array}{c}
\underrightarrow{\dom } \\
\overset{0}{\longleftarrow } \\
\overrightarrow{\cod }%
\end{array}%
\hom \left( A,B\right) ,
\end{equation}%
and \emph{actions}%
\begin{eqnarray*}
&&%
\begin{array}{ccc}
H\left( B,C\right) \times \hom \left( A,B\right) & \longrightarrow & H\left(
A,C\right) \\
\left( x,f\right) & \longmapsto & xf%
\end{array}
\\
&&%
\begin{array}{ccc}
\hom \left( B,C\right) \times H\left( A,B\right) & \longrightarrow & H\left(
A,C\right) \\
\left( g,y\right) & \longmapsto & gy%
\end{array}%
\end{eqnarray*}%
satisfying the following conditions%
\begin{gather}
\dom \left( gy\right) =g\dom \left( y\right) \ ,\ \dom \left(
xf\right) =\dom \left( x\right) f  \label{Naturality of dom,cod,0,+} \\
\cod \left( gy\right) =g\cod \left( y\right) \ ,\ \cod \left(
xf\right) =\cod \left( x\right) f  \notag \\
g0_{f}=0_{gf}=0_{g}f  \notag \\
\left( x+x^{\prime }\right) f=xf+x^{\prime }f\ ,\ g\left( y+y^{\prime
}\right) =gy+gy^{\prime }  \notag
\end{gather}%
\begin{eqnarray}
g^{\prime }\left( gy\right) &=&\left( g^{\prime }g\right) y\ ,\ \left(
xf\right) f^{\prime }=x\left( ff^{\prime }\right)  \label{Functoriality of H}
\\
g^{\prime }\left( xf\right) &=&\left( g^{\prime }x\right) f  \notag \\
1_{C}x &=&x=x1_{B}  \notag
\end{eqnarray}%
\begin{eqnarray}
\dom \left( 0_{f}\right) &=&f=\cod \left( 0_{f}\right)
\label{axioms for a category object} \\
\dom \left( x+x^{\prime }\right) &=&\dom(x^{\prime })\ ,\ \cod \left(
x+x^{\prime }\right) =\cod(x)  \notag \\
0_{\cod x}+x &=&x=x+0_{\dom x}  \notag \\
x+\left( x^{\prime }+x^{\prime \prime }\right) &=&\left( x+x^{\prime
}\right) +x^{\prime \prime }.  \notag
\end{eqnarray}
\end{proposition}

\begin{proof}
For every $f:A^{\prime }\longrightarrow A$, $g:B\longrightarrow B^{\prime }$
and $x\in H\left( A,B\right) $, write%
\begin{equation*}
H\left( f,g\right) \left( x\right) =gxf
\end{equation*}%
and it is clear that the set of conditions $\left( \ref{Naturality
of dom,cod,0,+}\right) $ asserts the naturality of
$\dom ,\cod ,0,+$; the set of conditions $\left(
\ref{Functoriality of H}\right) $ asserts the functoriality  of $H$ and the set of conditions $\left( \ref{axioms
for a category object}\right) $ asserts the axioms for an internal category.
\end{proof}

On a category $\C$ with a 2-cell structure $\mathbf{H}$ in general, there is no horizontal composition (or pasting) for 2-cells, that being the case only when the pair $(\C,\mathbf{H})$ is a 2-category.

\begin{proposition}
\label{Prop. 2-category as a cat with 2-cell stru + naturality}A category $%
\mathbf{C}$ with a 2-cell structure%
\begin{equation*}
\mathbf{H}=\left( H,\dom ,\cod ,0,+\right) ,
\end{equation*}%
is a 2-category if and only if the \emph{naturality condition}
\begin{equation}
\cod \left( x\right) y+x\dom \left( y\right) =\smallskip x\func{cod%
}\left( y\right) +\dom \left( x\right) y  
%\tag{naturality condition}
\label{Naturality condition for x,y}
\end{equation}%
holds for every $x\in H\left( B,C\right) ,\ y\in H\left( A,B\right) ,$ and
every triple of objects $\left( A,B,C\right) $ in $\mathbf{C}$, as displayed
in the diagram below%
\begin{equation*}
\xymatrix{
A
  \ar@/^1pc/[r]^{\dom y}="1"
  \ar@/_1pc/[r]_{\cod y }="2"
 &  B
  \ar@/^1pc/[r]^{\dom x}="3"
  \ar@/_1pc/[r]_{\cod x }="4"
 & C \\
\ar@{=>}^{ y } "1";"2"
\ar@{=>}^{ x } "3";"4"}.
\end{equation*}
\end{proposition}

\begin{proof}
If $(\C,\mathbf{H})$ is a 2-category, the naturality condition follows from the
horizontal composition of 2-cells and, conversely, given a 2-cell structure $\mathbf{H}$
over $\C$, in order to make it a 2-category one has to define a
horizontal composition which is given by%
\begin{equation*}
x\circ y=\cod \left( x\right) y+x\dom \left( y\right)
\end{equation*}%
or
\begin{equation*}
x\circ y=x\cod \left( y\right) +\dom \left( x\right) y
\end{equation*}%
provided the naturality condition is satisfied for every appropriate $x$ and $y$.
The middle interchange law follows from the naturality condition.
\end{proof}

The next section is dedicated to a list of working examples which will be used further on.

\section{Examples}\label{section 4}

In order to illustrate the notion of a category $\C$ with a 2-cell structure $\mathbf{H}$, in the notation of Proposition \ref{giving a 2-cell structure}, we consider some special examples of categories on which we can define certain kinds of 2-cell structures.

\subsection{The terminal category}\label{ex1}

In this case $\C=\mathbf{1}$ and $\hom(1,1)$ is a singleton set, which means that $\dom$ and $\cod$ are uniquely determined. Hence $H=H(1,1)$ is just a set and $\mathbf{H}=(H,0,+)$ is simply a monoid. Indeed all the conditions on  $(\ref{Naturality of dom,cod,0,+})$ are trivial, the conditions on $(\ref{Functoriality of H})$ force the actions to be trivial, while the last two conditions on $(\ref{axioms for a category object})$ ensure that $(H,0,+)$ is a monoid.

\subsection{A discrete category}\label{ex2}

In this case $\hom(A,B)$ is either empty or a singleton set, respectively if $A\neq B$ or $A=B$.
This means that a 2-cell structure on $\C$ is a collection of monoids $\mathbf{H}_A=(H,0,+)_{A}$, one for each object $A$ in $\C$.

\subsection{A codiscrete category}\label{ex3}

If $\C$ is a codiscrete category, that is, $\hom(A,B)$ is a singleton set for each pair of its objects, then to give a 2-cell structure is to specify, for each pair of objects $(A,B)$ in $\C$, a monoid $(H(A,B),0,+)$ and for every four-tuple of objects $(A,B,C,D)$ in $\C$ a homomorphism of monoids \[\mu(A,B,C,D)\colon{H(B,C)\to H(A,D)}\] such that for every $A',A,B,C,D,D'$, objects in $\C$, \[\mu(A',B,C,D')=\mu(A',A,D,D')\circ \mu(A,B,C,D)\] and \[\mu(A,A,B,B)=1_{H(A,B)}.\] In other words, we have to specify a functor $H\colon{\C^{op}\times \C\to \Mon}$ into the category $\Mon$ of monoids.

\subsection{A preorder considered as a category}\label{ex4}

When $\C$ is a preorder, considered as a category,  as it follows from the previous example, to specify a 2-cell structure over it is to give, for every two objects $A,B$ in $\C$ with $A\leq B$, a monoid $(H(A,B),0,+)$ and for every four objects $A,B,C,D$ in $\C$ with $A\leq B\leq C\leq D$, a homomorphism of monoids \[\mu(A,B,C,D)\colon{H(B,C)\to H(A,D)}\] such that for every $A'\leq A\leq B\leq C\leq D\leq D'$, \[\mu(A',B,C,D')=\mu(A',A,D,D')\circ \mu(A,B,C,D)\] and \[\mu(A,A,B,B)=1_{H(A,B)}.\] In other words it is a functor $H\colon{\C^{op}\times \C\to \Mon}$.

\subsection{A monoid considered as a one object category}\label{ex5}

When $\C$ is a monoid, considered as a one object category, the description given in Proposition \ref{giving a 2-cell structure}, for a 2-cell structure over $\C$, cannot be simplified much further except perhaps for saying that instead of a collection of internal categories $(\ref{intcat Hom(A,B)})$, one for each pair of objects $(A,B)$ in $\C$, there will be only one for the unique object of $\C$, in this particular case.

Nevertheless, we may consider some special classes of 2-cell structures that may be defined over $\C$.

If $\C=(M,1,\cdot)$ is a monoid (that is, a one object category) then every monoid $(H,0,+)$, on which $M$ acts on the left and on the right, together with a map\[D\colon{H\times M\to M}\] such that \begin{eqnarray}
D(0,f)&=&f\nonumber\\
D(x'+x,f)&=&D(x',D(x,f))\nonumber\\
gD(x,f)h&=&D(gxf,gfh)\nonumber
\end{eqnarray}
for every $x,x'\in H$ and $f,g,h\in M$, induces a 2-cell structure over $\C$. The needed diagram $(\ref{intcat Hom(A,B)})$ is obtained as
\[\xymatrix{H\times H\times M\ar[r]^{m}&H\times M\ar@<1.5ex>[r]^(.6){\pi_M}\ar@<-1.5ex>[r]_(.6){D}&M\ar[l]|(.35){\langle 0, 1\rangle}}\] 
with $m(x,y,f)=(x+y,f)$.

\subsection{A group considered as a one object groupoid}\label{ex6}

In the same fashion of the previous example, here we have $\C=(G,1,\cdot)$ which is now a group and for every other group $(H,0,+)$ on which $G$ acts on the left and on the right, together with a map $D\colon{H\times M\to M}$ satisfying the same three conditions from above, induces, in the same way as before, a 2-cell structure over $\C$. This example is important because in here we are considering all the 2-cells to be invertible (Definition \ref{def invcell}).

We will also introduce the notion of commutator (Definition \ref{def commutator}) for appropriate 2-cells, which in this case is always defined and it is given by the formula
\begin{equation*}
\lbrack (x,f),(y,g)]=D\left( y,g\right) x+yf-yD\left( x,f\right) -gx.
\end{equation*}

The commutator vanishes if and only if the 2-cells $x$ and $y$ can be composed horizontally (or pasted together).

\subsection{A category of the form $\Cat(\mathbf{B})$ for some category B}\label{Example: 2-cells
in Cat}\label{ex7}

When $\C$ is a category of internal categories in some category $\mathbf{B}$, we can always consider the 2-cell structure given by the internal transformations in the sense of section \ref{section2}.

If $A=\left( A_{0},A_{1},A_2,d,c,e,m,\pi_1,\pi_2\right)$ and  $B=\left( B_{0},B_{1},B_2,d',c',e',m',\pi'_1,\pi'_2\right)$ are two internal categories in $\B$, then  we define
\begin{eqnarray*}
H\left( A,B\right) &=&\left\{ \left( t_1,t_2,t_3\right) \mid
t_1\colon{A_0\to B_1},t_2,t_3\colon{A_1\to B_2}, \pi'_1 t_2=t_1c,\pi'_2t_3=t_1d\right\} \\
\end{eqnarray*}% 
and
\begin{eqnarray*}
\dom(A,B) \left( t\right) &=&(d't_1,\pi'_2 t_2) \\
\cod(A,B) \left( t\right) &=&(c't_1,\pi'_1 t_3).
\end{eqnarray*}%

If $f=(f_0,f_1)\colon{A'\to A}$ and $g=(g_0,g_1)\colon{B\to B'}$ are internal functors, we write $g_2$ to abbreviate $g_1\times_{g_0}g_1$ and then define
\begin{eqnarray*}
H\left( f,g\right)(t)=H((f_0,f_1),(g_0,g_1)) \left( t_1,t_2,t_3\right) &=&\left( g_1t_1f_0,g_{2}t_2f_{1},g_2t_3f_1\right).
\end{eqnarray*}%

To say what are the  identity two cells and how the vertical composition is defined, we first consider the set $H(A,B)$ as the set \[\{(t_1,h_1,k_1)\mid t_1\colon{A_0\to B1},h_1,k_1\colon{A_1\to B1}, (h_1,d't_1),(k_1,c't_1)\in\hom(A,B)\},\]
denoted by $L(A,B)$, which is in bijection with $H(A,B)$ as follows:
\begin{eqnarray}
(t_1,t_2,t_3)&\mapsto &(t_1,\pi'_2t_2,\pi'_1 t_3),\nonumber\\
(t_1,h_1,k_1)&\mapsto& (t_1,\langle t_1c,h_1\rangle,\langle k_1,t_1 d\rangle)\nonumber.
\end{eqnarray}
So if $h=(h_0,h_1)\colon{A\to B}$ is an internal functor and given any $t,s\in H(A,B)\cong L(A,B)$, such that $\dom(t)=h=\cod(s)$ we define
\begin{eqnarray}
0_{h} &\cong&( e'h_{0},h,h) \\
t +_h s &\cong&(m\langle t_1,s_1\rangle,\pi'_2s_2,\pi'_1 t_3).
\end{eqnarray}

Later on (Proposition \ref{prop_naturality of transformations} and subsequent) we will show that, in this context, every internal natural transformation is a natural 2-cell (in the sense of Definition \ref{def natcell}).

\subsection{The category of groups and group homomorphisms}\label{ex8}

%(This is a special case of example \ref{ex9} )

When $\mathbf{C}=\Grp$, the category of groups and group
homomorphisms, we may consider the
canonical 2-cells obtained by considering each group as a one
object groupoid and each group homomorphism as a functor. In that case, as
it is well known, a 2-cell
\begin{equation*}
t:f\longrightarrow g
\end{equation*}%
from the homomorphism $f$ to the homomorphism $g$, both from the group $A$
to the group $B$, is an element $t\in B$ such that%
\begin{equation*}
tf\left( x\right) =g\left( x\right) t\ \ \ ,\ \ \ \text{for all }x\in A.
\end{equation*}%
Since for a given $t$ and $f$, the homomorphism $g$ is uniquely determined as%
\begin{equation*}
g\left( x\right) =tf\left( x\right) t^{-1}=\ ^{t}f\left( x\right) ,
\end{equation*}%
this particular 2-cell structure over $\Grp$ is just an instance of example \ref{Example 1} (see below) with $\Grp$ instead of $\Mon$.

To see it we consider $H$, the functor that projects the second argument%
\begin{gather*}
H:\Grp^{op}\times \Grp\longrightarrow \Grp \\
\left( A,B\right) \longmapsto B
\end{gather*}%
together with%
\begin{gather*}
D:B\times \hom \left( A,B\right) \longrightarrow \hom \left( A,B\right) \\
\left( t,f\right) \longmapsto \ ^{t}f
\end{gather*}%
and it is a straightforward calculation to check that%
\begin{eqnarray*}
D\left( 0,f\right) &=&f \\
D\left( t+t^{\prime },f\right) &=&D\left( t,D\left( t^{\prime },f\right)
\right).
\end{eqnarray*}%

Moreover, since condition $\left( \ref{Example 1 nat cond}\right) $ is
satisfied (see below), the 2-cell structure is natural.

\subsection{Abstract 2-cells, and conjugations\label{Example 1}}\label{ex9}

This example is motivated from the examples \ref{ex5}, \ref{ex6} and \ref{ex8}.

Let \textbf{C} be a category and consider
\begin{equation*}
H:\mathbf{C}^{op}\times \mathbf{C}\longrightarrow \Mon
\end{equation*}%
a functor into $\Mon$, the category of monoids, together with a natural
transformation%
\begin{equation*}
D:UH\times \hom _{\mathbf{C}}\longrightarrow \hom _{\mathbf{C}}
\end{equation*}%
(with $U:\Mon\longrightarrow \Set$ denoting the usual forgetful functor) satisfying%
\begin{eqnarray*}
D\left( 0,f\right) &=&f \\
D\left( x^{\prime }+x,f\right) &=&D\left( x^{\prime },D\left( x,f\right)
\right)
\end{eqnarray*}%
for all $f:A\longrightarrow B$ in \textbf{C} and\textbf{\ }$x^{\prime },x\in
H\left( A,B\right) $, where $0$ denotes the zero element in the monoid $H\left( A,B\right) $, which is
written in additive notation although it is not necessarily commutative.\newline
With this information, $(H,D)$, we can define a 2-cell structure over \textbf{C}, which has 2-cells displayed as
\begin{equation*}
\cela {A} {f} {(x,f)} {D(x,f)} {B},%
\end{equation*}%
with vertical composition given as illustrated%
\begin{gather*}
\xymatrix{
A
  \ar@/^3pc/[rr]^f="1"
  \ar[rr]|{D(x,f)}="2"
  \ar@/_3pc/[rr]_{D(x',D(x,f)) }="3"
 &  & B
 \\
\ar@{=>}|{(x,f) } "1";"2"
\ar@{=>}|{(x',D(x,f)) } "2";"3"
}
\\
\left( x^{\prime },D\left( x,f\right) \right) +\left( x,f\right) =\left(
x^{\prime }+x,f\right),
\end{gather*}%
which is well defined because $D\left( x^{\prime }+x,f\right) =D\left( x^{\prime
},D\left( x,f\right) \right) $. The identity 2-cells are of the form%
\begin{equation*}
\cela {A} {f} {(0,f)} {f} {B}%
\end{equation*}%
which are well defined because $D\left( 0,f\right) =f$. The left and right actions
of the morphisms in the 2-cells,%
\begin{equation*}
\xymatrix{
A'
\ar[r]^h
& A
  \ar@/^1pc/[r]^f="1"
  \ar@/_1pc/[r]_{D(x,f) }="2"
 &   B
  \ar[r]^g
 & B' \\
\ar@{=>}|{(x,f) } "1";"2"}
\end{equation*}%
are given by the formulas
\begin{equation*}
g\left( x,f\right) h=\left( gxf,gfh\right) =\left( H\left( h,g\right) \left(
x\right) ,gfh\right) .
\end{equation*}%

If in addition%
\begin{equation}
D\left( y,g\right) x+yf=yD\left( x,f\right) +gx , \label{Example 1 nat cond}
\end{equation}%
for all $x,y,f,g$ as pictured
\begin{equation*}
\xymatrix{
A
  \ar@/^1pc/[r]^f="1"
  \ar@/_1pc/[r]_{D(x,f) }="2"
 &  B
  \ar@/^1pc/[r]^g="3"
  \ar@/_1pc/[r]_{D(y,g) }="4"
 & C \\
\ar@{=>}|{(x,f) } "1";"2"
\ar@{=>}|{(y,g) } "3";"4"},
\end{equation*}%
then the 2-cell structure is natural and the result is a 2-category.

\subsection{Abstract 2-cells, and derivations\label{Example 2}}\label{ex10}

This example is motivated from the example of crossed modules with derivations, as described in  \ref{ex11}.

Here $\C$ has to be a category for which the functor%
\begin{equation*}
\hom _{\mathbf{C}}:\mathbf{C}^{op}\times \mathbf{C}\longrightarrow \Set
\end{equation*}%
can be extended into $\Mon$, that is, there is a functor (denoted by $\map$ --- we are thinking  in the underlying map of a homomorphism)
\begin{equation*}
\map:\mathbf{C}^{op}\times \mathbf{C}\longrightarrow \Mon\overset{U}{%
\longrightarrow }\Set
\end{equation*}%
with $\hom \subseteq U\map$, in the sense that $\hom \left( A,B\right)
\subseteq U(\map\left( A,B\right) )$, naturally for every $A,B\in \mathbf{C}$.

When this is the case then any functor%
\begin{equation*}
K:\mathbf{C}^{op}\times \mathbf{C}\longrightarrow \Mon
\end{equation*}%
together with a natural transformation%
\begin{equation*}
D:K\longrightarrow \map,
\end{equation*}%
determines a 2-cell structure over $\C$ as follows.
The functor $H:\mathbf{C}^{op}\times \mathbf{C}\longrightarrow \Set$
 is given by
\begin{eqnarray*}
H\left( A,B\right) &=&\left\{ \left( x,f\right) \in UK\left( A,B\right)
\times \hom \left( A,B\right) \mid D\left( x\right) +f\in \hom \left(
A,B\right) \right\} \\
H\left( h,g\right) \left( x,f\right) &=&\left( gxh,gfh\right)
\end{eqnarray*}%
while $\dom ,\cod ,0,+,$ are given as illustrated 
\begin{equation*}
\xymatrix{
A
  \ar@/^2pc/[rr]^f="1"
  \ar@/_2pc/[rr]_{D(x)+f }="2"
 & &  B,
 \\
\ar@{=>}|{(x,f) } "1";"2"
}%
\end{equation*}%
with $\left( x,f\right) \in H\left( A,B\right)$; the formula for the vertical composition is $$\left( x^{\prime },D\left( x\right) +f\right) +\left(
x,f\right) =\left( x^{\prime }+x,f\right); $$ the 
identity cells are of the form $$\left( 0,f\right); $$ while the left and right actions are computed as $$g\left( x,f\right) h=\left( gxh,gfh\right). $$

If in addition the property
\begin{equation}
D\left( y\right) x+gx+yf=yD\left( x\right) +yf+gx  \label{cond *1}
\end{equation}%
is satisfied for all $\left( x,f\right) \in H_{}\left( A,B\right) $ and $%
\left( y,g\right) \in H_{}\left( A,C\right) $, then the resulting structure
is a 2-category.

\subsection{The case of crossed modules}\label{ex11}

In the case \textbf{C}=X-Mod, the category of crossed modules, we have
the canonical 2-cell structure given by derivations, which is an instance
of the above construction with $\Grp$ instead of $\Mon$:

The objects in X-Mod are of the form%
\begin{equation*}
A=\left( X\overset{d}{\longrightarrow }B,\varphi :B\longrightarrow Aut\left(
X\right) \right)
\end{equation*}%
in which $d:X\longrightarrow B$ is a group homomorphism, $\varphi$ is a group
action of $B$ in $X$, denoted by $\varphi(b)(x)=b\cdot x$, and satisfying
\begin{eqnarray*}
d\left( b\cdot x\right) &=&bd\left( x\right) b^{-1} \\
d\left( x\right) \cdot x^{\prime } &=&x+x^{\prime }-x.
\end{eqnarray*}%
A morphism $f:A\longrightarrow A^{\prime }$ in X-Mod is of the form%
\begin{equation*}
f=\left( f_{1},f_{0}\right)
\end{equation*}%
with $f_{1}:X\longrightarrow X^{\prime }$ and $f_{0}:B\longrightarrow
B^{\prime }$  group homomorphisms such that
\begin{equation*}
f_{0}d=d^{\prime }f_{1}
\end{equation*}%
and
\begin{equation*}
f_{1}\left( b\cdot x\right) =f_{0}\left( b\right) \cdot f_{1}\left( x\right)
.
\end{equation*}%
Clearly there are functors%
\begin{equation*}
\map:\mathbf{C}^{op}\times \mathbf{C}\longrightarrow \Grp
\end{equation*}%
sending $\left( A,A^{\prime }\right) $ to the group of pairs $\left(
f_{1},f_{0}\right) $ of maps (not necessarily homomorphisms) $%
f_{1}:UX\longrightarrow UX^{\prime }$ and $f_{0}:UB\longrightarrow
UB^{\prime }$ such that%
\begin{equation*}
f_{0}d=d^{\prime }f_{1},
\end{equation*}%
with the group operation defined componentwise%
\begin{equation*}
\left( f_{1},f_{0}\right) +\left( g_{1},g_{0}\right) =\left(
f_{1}+g_{1},f_{0}+g_{0}\right) .
\end{equation*}%
Moreover, there is a functor%
\begin{equation*}
M:\mathbf{C}^{op}\times \mathbf{C}\longrightarrow \Grp
\end{equation*}%
sending $\left( A,A^{\prime }\right) $ to the group $M\left( A,A^{\prime
}\right) =\left\{ t\mid t:UB\longrightarrow UX^{\prime }\ \text{is a map}%
\right\} $, and a natural transformation%
\begin{equation*}
D:M\longrightarrow \map
\end{equation*}%
defined by
\begin{equation*}
D\left( A,A^{\prime }\right) \left( t\right) =\left( td,dt\right) .
\end{equation*}%
With this data we define $H\left( A,A^{\prime }\right) $ as%
\begin{equation*}
\left\{ \left( t,f\right) \mid t\in M\left( A,A^{\prime }\right) \ ,\
f=\left( f_{1},f_{0}\right) :A\longrightarrow A^{\prime }\ ,\ \left(
td+f_{1},dt+f_{0}\right) \in \hom \left( A,A^{\prime }\right) \right\} .
\end{equation*}%
It is well known that the map $t:B\longrightarrow X^{\prime }$ is such that
\begin{equation*}
t\left( bb^{\prime }\right) =t\left( b\right) +f_{0}\left( b\right) \cdot
t\left( b^{\prime }\right) \ \ ,\ \ \text{for all }b,b^{\prime }\in B\text{,}
\end{equation*}%
while $\left( td+f_{1},dt+f_{0}\right) \in \hom \left( A,A^{\prime }\right) $
asserts that the pair $\left( td+f_{1},dt+f_{0}\right) $ is a morphism of
crossed modules%
\begin{equation}
\quadrado
{ X  }       { d  }     {   B  }
{ td+f_1  }                  {  dt+f_0   }
{ X'  }       {  d }     {    B' }
\label{square}
\end{equation}%
and it is equivalent to

\begin{itemize}
\item $dt+f_{0}$ is a homomorphism of groups%
\begin{equation*}
dt\left( bb^{\prime }\right) =d\left( t\left( b\right) +f_{0}\left( b\right)
\cdot t\left( b^{\prime }\right) \right)
\end{equation*}

\item $td+f_{1}$ is a homomorphism of groups%
\begin{equation*}
t\left( d\left( x\right) d\left( x^{\prime }\right) \right) =t\left(
dx\right) +f_{0}d\left( x\right) \cdot td\left( x^{\prime }\right)
\end{equation*}

\item the square $\left( \ref{square}\right) $ commutes, which is trivial
because $\left( f_{1},f_{0}\right) \in \hom \left( A,A^{\prime }\right) $

\item $\left( td+f_{1}\right) $ preserves the action of $\left(
dt+f_{0}\right) $%
\begin{equation*}
t\left( bd\left( x\right) b^{-1}\right) =t\left( b\right) +f_{0}\left(
b\right) \cdot t\left( d\left( x\right) \right) +f_{0}\left( bd\left(
x\right) b^{-1}\right) \cdot \left( -t\left( b\right) \right) .
\end{equation*}
\end{itemize}
Thus, this is an instance of the example presented in \ref{ex10}. This particular example is further explored in \cite{NMF-bicat} where the description of pseudocategory given in Theorem~\ref{thm WMC} is further detailed.

\subsection{Abstract 2-cells and homotopies\label{Remark20}}\label{ex12}

A particular (but important) case, which is obtained from the more general example presented in \ref{ex10}, is when $\mathbf{C}$ is an Ab-category. A 2-Ab-category, as defined
in \cite{Ch5} and \cite{Ch6}, is obtained in this way. In that case the
functor $\hom $ coincides with $\map$ and moreover it is a functor into $\Ab$, the category of abelian groups.
\begin{equation*}
\hom :\mathbf{C}^{op}\times \mathbf{C}\longrightarrow \Ab
\end{equation*}%
This means that giving a 2-cell structure on $\C$ is to give a functor  $$H:\mathbf{C}^{op}\times \mathbf{C}\longrightarrow \Ab,$$which is usually required to be an Ab-functor, together with a
natural transformation $D:H\longrightarrow \hom $. The 2-cell structure thus obtained
makes $\mathbf{C}$ into a 2-category (in fact a 2-Ab-category) if and only if the
condition  $D\left( y\right) x=yD\left( x\right) $ is satisfied for every appropriate $x$ and $y$. This condition is just how the condition $\left( \ref{cond *1}\right) $ simplifies in this context, see \cite{Ch5} and \cite{Ch6} for more details. Furthermore,
as proved in \cite{Ch5}, every 2-cell structure (if enriched in $\Ab$) is
obtained in this way.

Many of these considerations are still valid for an arbitrary monoidal category $%
\mathbf{V}$ instead of $\Ab$, but in general not all 2-cell structures are obtained in this way.

\subsection{The case of Abelian Chain Complexes}\label{ex13}

We now consider an example of a category with a \emph{canonical} 2-cell structure in which not every 2-cell is a natural 2-cell (see Definition \ref{def natcell}).

The example of abelian chain complexes (say of order 2 to simplify notation)
is self explanatory (see also \cite{Bourn},\cite{BR} and references there).
We have objects, morphisms and 2-cells (homotopies) as displayed%
\begin{equation}
\xymatrix{ A_2 \ar[r]^{d} \ar@<.5ex>[d]^{g_2} \ar@<-.5ex>[d]_{f_2} & A_1
\ar[r]^{d} \ar@<.5ex>[d]^{g_1} \ar@<-.5ex>[d]_{f_1} \ar[ld]_{t_2} & A_0
\ar@<.5ex>[d]^{g_0} \ar@<-.5ex>[d]_{f_0} \ar[ld]_{t_1} \\ A'_2 \ar[r]^{d} &
A'_1 \ar[r]^{d} & A'_0 }  \label{diagram t1.t2}
\end{equation}%
with%
\begin{eqnarray*}
dt_{1} &=&g_{0}-f_{0} \\
t_{1}d+dt_{2} &=&g_{1}-f_{1} \\
t_{2}d &=&g_{2}-f_{2}
\end{eqnarray*}%
or equivalently%
\begin{eqnarray*}
g_{0} &=&dt_{1}+f_{0} \\
g_{1} &=&t_{1}d+dt_{2}+f_{1} \\
g_{2} &=&t_{2}d+f_{2}
\end{eqnarray*}%
and hence we have, for $\mathbf{C}=$2-Ch($\Ab$), the functor%
\begin{equation*}
H:\mathbf{C}^{op}\times \mathbf{C}\longrightarrow \Ab
\end{equation*}%
sending the pair of objects $\left( A,A^{\prime }\right) $ to the abelian
group of pairs $\left( t_{2},t_{1}\right) $; and the natural transformation%
\begin{equation*}
D:H\longrightarrow \hom
\end{equation*}%
sending a pair $\left( t_{2},t_{1}\right) $ as above to the triple $\left(
t_{2}d,t_{1}d+dt_{2},dt_{1}\right) $ displayed as follows%
\begin{equation*}
\xymatrix{ A_2 \ar[r]^{d} \ar[d]_ {t_2 d} & A_1 \ar[r]^{d} \ar[d]| {t_1 d+d
t_2} & A_2 \ar[d]_ {d t_1} \\ A'_2 \ar[r]^{d} &A'_1 \ar[r]^{d} & A'_0 }.
\end{equation*}%
This is clearly an instance of the construction used in examples \ref{ex10} and \ref{ex13}, however, condition%
\begin{equation*}
D(x)y=xD(y)
\end{equation*}%
is not always satisfied, since it becomes, for $x=\left(
t_{2},t_{1}\right) $ and $y=\left( s_{2},s_{1}\right) $%
\begin{equation*}
\left( t_{2}ds_{2},dt_{2}s_{1}+t_{1}ds_{1}\right) =\left(
t_{2}ds_{2}+t_{2}s_{1}d,t_{1}ds_{1}\right)
\end{equation*}%
which holds if $t_{2}s_{1}=0$, but not in general.

The commutator (see Definition \ref{def commutator}) in this case is given by the formula%
\begin{equation*}
\lbrack x,y]=\left( -t_{2}s_{1}d,dt_{2}s_{1}\right),
\end{equation*}
whenever it is defined.

\subsection{A slightly more general example}\label{ex14}

This is still a slightly further generalization of the example used in  \ref{ex10}.

Suppose  $\mathbf{A}$ is a category admitting a \emph{forgetful} functor into
$\Set$, say $U:\mathbf{A}\longrightarrow \Set$, and let us assume the existence of a functor%
\begin{equation*}
\map:\mathbf{C}^{op}\times \mathbf{C}\longrightarrow \mathbf{A}
\end{equation*}%
such that
\begin{equation*}
\hom _{C}\left( A,B\right) \subseteq U(\map\left( A,B\right))
\end{equation*}%
as detailed  in Subsection \ref{ex10} but with a generic $\mathbf{A}$ instead of $\Mon$.

 Once having such data we
may be interested in considering the 2-cell structures over $\mathbf{C}$
that are \emph{loosely enriched} in $\mathbf{A}$ in the same way
as $\mathbf{C}$ is. To do that we consider the 2-cell
structures given by%
\begin{equation*}
H\left( A,B\right) =\left\{ x\in UM\left( A,B\right) \mid U\dom x,U%
\cod x\in \hom \left( A,B\right) \right\}
\end{equation*}%
with $M,\dom ,\cod $ being part of an internal category object
in $\mathbf{A}^{\mathbf{C}^{^{op}}\times \mathbf{C}}$, of the form%
\begin{equation*}
M\times _{\map}M\overset{+}{\longrightarrow }M%
\begin{array}{c}
\underrightarrow{\dom } \\
\overset{0}{\longleftarrow } \\
\overrightarrow{\cod }%
\end{array}%
\map,
\end{equation*}%
with the obvious restrictions after applying $U$.

It is now interesting to observe that the case $\mathbf{A}=\Mon$ is precisely the construction of Section \ref{Example 2}. Another interesting case is when  $\mathbf{A}=\Grp$. If $%
\mathbf{A}=\Ab$ and also asking for $M$ to be an Ab-functor, then the result
is a 2-Ab-category if also adding the condition%
\begin{equation*}
D\left( x\right) y=xD\left( y\right)
\end{equation*}%
for all appropriate $x$ and $y$ (example \ref{ex12}).

When $\mathbf{A}$ is a monoidal category and $\mathbf{C}$ a category enriched in $%
\mathbf{A}$ then we can always find the functor $\map$ in the conditions above by putting $\map=\hom $. 

%It is then reasonably to say that, in this more
%general context, the category $\mathbf{C}$ is weakly enriched in $\mathbf{A}$
%(for example, in this sense, Groups are weakly enriched in Groups, and every
%algebraic structure is weakly enriched in itself).

\subsection{Topological Abelian Groups}\label{ex15}

When $\mathbf{C}=TAG$, is the category of topological abelian groups,  we can always consider  the
2-cell structure given as follows: for every topological abelian groups $X$ and $Y$, $H\left( X,Y\right)$ is the quotient set
\begin{equation*}
\left\{ \alpha :I\times X\longrightarrow Y\mid \alpha
\text{ is continuous, }\alpha \left( 0,\_\right) =0\text{ , }\alpha \left(
t,\_\right) \text{ is a homomorphism}\right\} /\sim
\end{equation*}%
in which $I$ denotes the unit interval and the equivalence $\sim $ identifies
\begin{equation*}
\alpha \sim \beta
\end{equation*}%
if and only if:
\begin{enumerate}
\item  $\alpha \left( 1,\_\right) =\beta \left( 1,\_\right) $, which is  referred to as $h$ in item 2. just below,
\item there exits $\Phi :I\times I\times X\longrightarrow Y$, continuous and such that%
\begin{eqnarray*}
\Phi \left( 0,\_,\_\right) &=&\alpha \\
\Phi \left( 1,\_,\_\right) &=&\beta \\
\Phi \left( s,0,\_\right) &=&0 \\
\Phi \left( s,1,\_\right) &=&h
\end{eqnarray*}%
\item and $\Phi \left( s,t,\_\right) $ is a homomorphism.
\end{enumerate}
The natural transformation $D:H\longrightarrow \hom $ is given by%
\begin{equation*}
D\left( [\alpha ]\right) =\alpha \left( 1,\_\right)
\end{equation*}%
with%
\begin{equation*}
\left( g[\alpha ]f\right) \left( t,x\right) =g\left( \alpha \left( t,f\left(
x\right) \right) \right) .
\end{equation*}%
It is clear that the condition
\begin{equation*}
\lbrack \alpha ]D\left( [\beta ]\right) =D\left( [\alpha ]\right) [\beta ]
\end{equation*}%
holds because%
\begin{equation*}
\alpha D\left( \beta \right) \sim D\left( \alpha \right) \beta
\Leftrightarrow \alpha \left( t,\beta \left( 1,x\right) \right) \sim \alpha
\left( 1,\beta \left( t,x\right) \right)
\end{equation*}%
and there exits
\begin{equation*}
\Phi \left( s,t,x\right) =\alpha \left( t^{\left( 1-s\right) },\beta \left(
t^{s},x\right) \right).
\end{equation*}

This means that we may use the previous information to build on TAG a 2-cell structure, as in example \ref{ex14}, knowing that the result is a 2-category.

Finally we observe that the notion of a category with a 2-cell structure, besides giving a simple
characterization of a 2-category as
\begin{equation*}
\text{\textquotedblleft 2-category"=\textquotedblleft
sesquicategory"+\textquotedblleft naturality condition",}
\end{equation*}%
it also gives a useful tool in generating examples for arbitrary situations.

\section{Morphisms between 2-cell structures, naturally invertible 2-cell structures and cartesian squares}\label{section 5}

For a fixed category $\C$, we consider the  category Cell(\textbf{C%
}) of all the possible 2-cell structures over $\C$, as well as its full subcategories NatCell($\C$), of natural 2-cell structures, and  InvCell($\C$), consisting of the 2-cell structures with invertible 2-cells. The
category Cell($\C$) has a initial object (the discrete 2-cell
structure) and a terminal object (the codiscrete 2-cell structure). In many cases (example \ref{ex7}, \ref{ex8}, \ref{ex11}, \ref{ex13}, etc.) a \emph{canonical non-trivial} 2-cell structure is also present. One particular case of our interest is $\C=\Cat(\mathbf{B})$.

%In some cases Cell($\C$)=NatCell($\C$) In the case of \ref{ex7}, for instance, there is always the 2-cell structure given by the internal transformations and the natural 2-cell
%structure given by of internal natural transformations.

\subsection{The category of 2-cell structures over a fixed category \textbf{C%
}}

In this section $\C$ is an arbitrary category which is fixed. 

The category Cell($\C$) has as objects the 2-cell structures over $\C$. If $$\mathbf{H}=( H,\dom ,\cod ,0,+)$$ and $$\mathbf{H'}=( H',\dom' ,\cod' ,0',+')$$ are two 2-cell structures over $\C$, then a morphism $\varphi :\mathbf{H}\longrightarrow \mathbf{H}^{\prime }$ is a
natural transformation
\begin{equation*}
\varphi :H\longrightarrow H^{\prime }
\end{equation*}%
such that
\begin{eqnarray*}
\dom ^{\prime }\varphi &=&\dom  \\
\cod ^{\prime }\varphi &=&\cod  \\
\varphi 0 &=&0^{\prime } \\
\varphi + &=&+^{\prime }\left( \varphi \times \varphi \right) .
\end{eqnarray*}%
We will sometimes write $H$ instead of $( H,\dom ,\cod ,0,+)$ to refer to a 2-cell structure.

The reason for describing Cell($\C$), the category of all 2-cell structures over $\C$, is the study of
pseudocategories internal to $\C$. A pseudocategory (Definition \ref{def pseudocat}), internal to 
$\C$, depends on the 2-cell structure which is being considered over $\C$. For
example, a pseudocategory in \textbf{C} with the codiscrete 2-cell structure
is a precategory, while if considering the discrete 2-cell structure it is
an internal category. It seems to be interesting, for a fixed
category $\C$, the study of the variation of a pseudocategory (from a
precategory to an internal category) as  the 2-cell structure considered over $\C$ varies from the initial to the terminal object in Cell($\C$). Another seemingly important aspect of Cell($\C$) is the following.
Every morphism%
\begin{equation}
\varphi :H\longrightarrow H^{\prime }  \label{phi}
\end{equation}%
in Cell(\textbf{C}) induces a functor%
\begin{equation}
PsCat\left( \mathbf{C},H\right) \longrightarrow PsCat\left( \mathbf{C}%
,H^{\prime }\right)  \label{induced phi}
\end{equation}%
from pseudocategories in \textbf{C}, relative to the 2-cell structure $H$, to
pseudocategories in \textbf{C}, relative to the 2-cell structure $H^{\prime }$.
We reserve for a future work the study of equivalent 2-cell
structures, by identifying $H$ and $H'$, via  $\left( \ref{phi}\right) $, 
whenever $\left( \ref{induced phi}\right) $ is an equivalence of categories, relating it with homotopy theory.

The notion of a pseudocategory, as defined in \cite{Ch1} (see also \cite{Ch5,Ch6}), rests on the
existence of some induced 2-cells between certain pullbacks. This means that those pullbacks have to share their universal property with morphisms and with 2-cells, thus the following definition.

\subsection{Cartesian $H$-squares}

We continue to assume that $\C$ is an arbitrary category, which is fixed, and now consider a functor $H\colon{\C^{op}\times \C\to \Set}$ for which we write $H(f,g)(x)$ as $gxf$. 

\begin{definition}[cartesian $H$-square]\label{Def: cartesian 2-cell structure, Ch2}\label{def cartesian} A commutative square \[\xymatrix{D\ar[r]^{\pi_2}\ar[d]_{\pi_1}&C\ar[d]^{g}\\A\ar[r]^{f}&B}\]in $\C$ is said to be $H$-cartesian if for every object $Z$ in $\C$, for every  $x\in H(Z,A)$ and $y\in H(Z,C)$ with $fx=gy$, there is a unique element $w\in H(z,D)$ such that $\pi_1w=x$ and $\pi_2w=y$. In that case we write $w$ as $\langle x,y\rangle$.
\end{definition}

An immediate observation is that when $H=\hom_{\C}$ then a square is $H$-cartesian if and only if it is a pullback square.

It is also not difficult to observe that when $H(D,-)$ preserves pullbacks for every object $D$ in $\C$, then every pullback square is $H$-cartesian. 

For that reason we will be interested in  considering 2-cell structures of which the functor $%
H\left( D,-\right) :\mathbf{C}\longrightarrow \Set$ preserves pullbacks for
every object $D$ in $\mathbf{C}$. This means that 
\begin{equation*}
H:\mathbf{C}^{op}\times \mathbf{C}\longrightarrow \Set,
\end{equation*}%
giving a 2-cell structure to a category $\mathbf{C}$, has the property that
\begin{equation*}
H\left( D,A\times _{\left\{ f,g\right\} }B\right) \overset{\varphi }{\cong }%
\left\{ \left( x,y\right) \in H\left( D,A\right) \times H\left( D,B\right)
\mid fx=gy\right\}
\end{equation*}%
for every object $D$ in $\mathbf{C}$ and pullback square
\begin{equation*}
\xymatrix{ A \times_C B \ar[r]^{ \pi_2 } \ar[d]_{ \pi_1 } & B \ar[d]^{ f }
\\ A \ar[r]^{ g } & C },
\end{equation*}%
with $\varphi $ a natural isomorphism, that is, for every
$h:D\longrightarrow D^{\prime }$, the following square commutes%
\begin{equation*}
\quadrado
{ H(D,A\times_C B)  }       { \cong \varphi  }     {   \{(x,y) \mid fx=gy \}  }
{ H(h,1)  }                  {     }
{  H(D',A\times_C B) }       { \cong \varphi  }     {  \{(x',y') \mid fx'=gy' \}   }%
\end{equation*}%
or, in other words, that%
\begin{equation*}
\left\langle x,y\right\rangle h=\left\langle xh,yh\right\rangle
\end{equation*}%
as displayed.
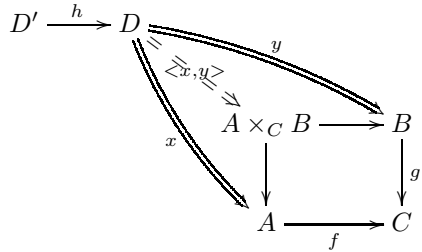
\begin{equation*}
\xymatrix{
D'
\ar[r]^{h}
& D
\ar@/^/@{=>}[rrd]^{y}
\ar@{==>}[dr]|-{<x,y>}
\ar@/_/@{=>}[ddr]_{x}
\\
& & A\times_C B
\ar[r]
\ar[d]
& B
\ar[d]^g
\\
& & A
\ar[r]_f
& C
\\
}
\end{equation*}

In particular, for $D=A^{\prime }\times _{C^{\prime }}B^{\prime }$, and
appropriate $x,y,z$ as in%
\begin{equation*}
\xymatrix{
A'
\ar@{=>}[d]_x
& A'\times_{C'} B'
\ar[l]_{\pi'_1}
\ar[r]^{\pi'_2}
\ar@{=>}[d]|{x\times_z y}
& B'
\ar@{=>}[d]^y
& , & C'
\ar@{=>}[d]^z
\\
A
& A\times_C B
\ar[l]_{\pi_1}
\ar[r]^{\pi_2}
& B
& , & C
\\
},
\end{equation*}%
the element $x\times _{z}y$ appears as the unique 2-cell in $H\left(
A^{\prime }\times _{C^{\prime }}B^{\prime },A\times _{C}B\right) $ such that
\begin{eqnarray*}
\pi _{2}\left( x\times _{z}y\right) &=&y\pi _{2}^{\prime } \\
\pi _{1}\left( x\times _{z}y\right) &=&y\pi _{1}^{\prime }.
\end{eqnarray*}

These observations will be used in Definition  \ref{def pseudocat}, saying what a pseudocategory internal to a category with a 2-cell structure is.
%
%The notion of $H$-cartesian square will be used in the definition of pseudocateory (see below) where the pullback square obtained from domain and codomain is also required to be $H$-cartesian, with respect to a given 2-cell structure. This will be done in the next section, for the moment we continue to analyse some other aspects of 2-cell structures.

The two other important aspects of a 2-cell structure, that are needed for the definition of pseudocategory, are the naturality and the invertibleness of the 2-cells in the 2-cell structure. Hence we are continuing with this analysis before going to the next section and define a internal pseudocategory, which is one of the main purposes of this work.

%Let $\mathbf{C}$ be a category.
%
%\begin{definition}[cartesian 2-cell structure]
%\label{Def: cartesian 2-cell structure, Ch2}A 2-cell structure $\left( H,%
%\dom ,\cod ,0,+\right) $ over the category $\mathbf{C}$ is said to
%be \emph{Cartesian} if the functor $H\left( D,\_\right) :\mathbf{C}%
%\longrightarrow Set$ preserves pullbacks for every object $D$ in $\mathbf{C}$%
%.
%\end{definition}

\subsection{Natural and invertible 2-cells on a 2-cell structure}

Recall from Proposition \ref{Prop. 2-category as a cat with 2-cell stru + naturality} that a category $%
\mathbf{C}$ with a 2-cell structure%
\begin{equation*}
\mathbf{H}=\left( H,\dom ,\cod ,0,+\right) ,
\end{equation*}%
is a 2-category if and only if the naturality condition $(\ref{Naturality condition for x,y})$
\begin{equation}
\cod \left( x\right) y+x\dom \left( y\right) =\smallskip x\func{cod%
}\left( y\right) +\dom \left( x\right) y  \nonumber
%\tag{naturality condition}
\end{equation}%
holds for every $x\in H\left( B,C\right) ,\ y\in H\left( A,B\right) ,$ and
every triple of objects $\left( A,B,C\right) $ in $\mathbf{C}$, as displayed
\begin{equation*}
\xymatrix{
A
  \ar@/^1pc/[r]^{\dom y}="1"
  \ar@/_1pc/[r]_{\cod y }="2"
 &  B
  \ar@/^1pc/[r]^{\dom x}="3"
  \ar@/_1pc/[r]_{\cod x }="4"
 & C \\
\ar@{=>}^{ y } "1";"2"
\ar@{=>}^{ x } "3";"4"}.
\end{equation*}
 
It may happen that the naturality condition does not hold for all the
possible  $x$ and $y$, but only for a few (as it is the case in \ref{ex6}, \ref{ex7}, \ref{ex9}, etc.). Thus the following definitions sprung.

Let $\mathbf{C}$ be a category and $\left( H,\dom ,\cod %
,0,+\right) $ a 2-cell structure over it.

\begin{definition}\label{def nat2cellwrt}
A 2-cell $y \in H\left( A,B\right) $ is said to be \emph{natural with respect to}
a 2-cell $z\in H\left( X,A\right) $, if
\begin{equation*}
\cod \left( y \right) z+y \dom \left( z\right) =y
\cod \left( z\right) +\dom \left( y \right) z.
\end{equation*}%
\end{definition}
When that is the case it means that the usual horizontal composition (or pasting) $y \circ z$ is well defined and it is given by each one of the formulas involved in the naturality condition.
\begin{definition}\label{def natcell}
%A 2-cell $\delta \in H\left( A,B\right) $ is \emph{natural} when it is
%natural with respect to all possible $z\in H\left( X,A\right) $ for all $%
%X\in \mathbf{C}$, i.e., $\delta $ is a natural 2-cell if and only if $\delta
%\circ z$ for all possible $z$.
A 2-cell $x\in H\left( A,B\right) $ is natural when%
\begin{equation}
\cod \left( x\right) y+x\dom \left( y\right) =x\cod \left(
y\right) +\dom \left( x\right) y  \label{eq_a}
\end{equation}%
for every element $y\in H\left( X,A\right) $ for every object $X$ in \textbf{%
C}.
\end{definition}
In other words, a 2-cell $x$ is natural when it is natural with respect to all the possible 2-cells with which it makes sense to horizontally compose it.
\begin{definition}\label{def invcell}

A 2-cell $x\in H\left( A,B\right) $ is invertible when there is a
(necessarily unique) element%
\begin{equation*}
-x\in H\left( A,B\right)
\end{equation*}%
such that $\dom \left( x\right) =\cod \left( -x\right) \ ,\ \func{%
cod}\left( x\right) =\dom \left( -x\right) $ and
\begin{equation*}
x+\left( -x\right) =0_{\cod \left( x\right) }\ ,\ \left( -x\right)
+x=0_{\dom \left( x\right) }.
\end{equation*}%
\end{definition}
The notion of a 2-cell, being natural with respect to another 2-cell, in the case when both the 2-cells are invertible, can be conveniently translated into the notion of a binary commutator, which is borrowed from the usual notion of commutator, in the sense of Group Theory.
\begin{definition}\label{def commutator} Let $x\in H(B,C)$ and $y\in H(A,B)$ be two invertible 2-cells in a 2-cell structure. The commutator 2-cell of $x$ and $y$, denoted by $[x,y]\in H(A,C)$
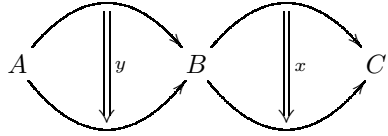
\begin{equation*}
\xymatrix{
A
  \ar@/^2pc/[rr]^{}="1"
  \ar@/_2pc/[rr]_{ }="2"
 & &  B
  \ar@/^2pc/[rr]^{}="3"
  \ar@/_2pc/[rr]_{ }="4"
 & & C \\
\ar@{=>}^{y } "1";"2"
\ar@{=>}^{x } "3";"4"}
\end{equation*}%
is given by the formula
\begin{eqnarray*}
\lbrack x,y] &=&\left( c_{1}+d_{2}-d_{1}-c_{2}\right) \left( x,y\right)  \\
&=&c_{1}\left( x,y\right) +d_{2}\left( x,y\right) -d_{1}\left( x,y\right)
-c_{2}\left( x,y\right)
\end{eqnarray*}%
with
\begin{eqnarray*}
c_{1}\left( x,y\right)  &=&\cod \left( x\right) y\ ,\ c_{2}\left(
x,y\right) =x\cod \left( y\right)  \\
d_{1}\left( x,y\right)  &=&\dom \left( x\right) y\ ,\ d_{2}\left(
x,y\right) =x\dom \left( y\right) .
\end{eqnarray*}%
\end{definition}

The comparison of the commutator with the 2-cell $0_{\cod \left( x\right) \cod \left(
y\right) }$ tell us the obstruction that $x$ and $y$ offer to be horizontally composed. In other words, the commutator vanishes, \[[x,y]=0_{\cod \left( x\right) \cod \left(
y\right) },\] if and only if $x$ is natural with respect to $y$.

%\hrulefill
%\begin{quotation}
%There is something strange here, because it also has to work for  or $0_{\dom \left( x\right) \dom \left(y\right) }$ , maybe defining $[x,y]=d_2-c_1+d_1-c_2$ or something like this.
%\end{quotation}
%\hrulefill

%[Reformulate the following text which give some remarks on abstract 2-cells and commutators in general]
%
%Previous examples apply to arbitrary (even large) categories, provided they
%admit the functors and the natural transformations as specified. Interesting
%examples also appear if one tries to particularize the category $\mathbf{C}$%
%. For example if \textbf{C} has only one object, or if it is a preorder; the
%first case gives something that particularizes to a (strict) monoidal
%category (with fixed set of objects) in the presence of the naturality
%condition; while the second case gives something that particularizes to an
%enriched category over monoids.\newline

We may thus consider the full subcategories of Cell($\C$) by taking the 2-cell structures where all the 2-cells are natural or invertible or both, denoted, respectively, by NatCell($\C$), InvCell($\C$) and NatInvCell($\C$).

In the simplest case, when $\mathbf{C}$=$\mathbf{1}$, considered in example \ref{ex1}, we have that \[\text{Cell($\C$)}=\Mon,\] NatCell($\C$) is the category of commutative monoids and InvCell($\C$) the category of groups. The case NatInvCell($\C$) being abelian groups.

This observation suggests the study of a generalization for  the well known reflection%
\begin{equation*}
\Grp\overset{I}{\longrightarrow }\Ab,
\end{equation*}
namely
\begin{equation*}
I\colon{\text{Cell($\C$)}\to \text{NatCell($\C$)}}
\end{equation*}%
from the category of 2-cell structures over $\mathbf{C}$, into the
subcategory of natural 2-cell structures over $\mathbf{C}$, associating to each
2-cell structure its \emph{naturalization}. This study is postponed for a future work. A simple observation may however be given in the case of $\mathbf{C}$ being an Ab-category (see \cite%
{Ch5},\cite{Ch6} and Section \ref{Remark20} above) which is that in this case the notion of commutator
reduces to%
\begin{equation*}
\lbrack x,y]=D\left( x\right) y-xD\left( y\right) .
\end{equation*}%
%In fact the notion of 2-Ab-category (as introduced in \cite{Ch5}) may be
%pushed further in the direction of a sesquicategory enriched in any category
%$\mathbf{A}$ with a \textquotedblleft forgetful" functor into Sets.

We shall now see how the notions of naturallity are related when $\C$ is of the form $\Cat\left( \mathbf{B}\right)$, for some category $%
\mathbf{B}$, with the 2-cell structure given by the internal (natural) transformations.

Before stating the result we refer to Section \ref{section2} observing that, for every internal category $A=\left( A_{0},A_{1},d,c,e,m\right)$,
there exists an internal category $$A^{\rightarrow }=\left( A_{1},A_{1},1,1,1,1\right),$$ which is in some sense the category of arrows of $A$, together with two internal functors
\begin{equation*}
d^{\rightarrow }=\left( ed,d\right) :A^{\rightarrow }\longrightarrow A
\end{equation*}%
and
\begin{equation*}
c^{\rightarrow }=\left( ec,c\right) :A^{\rightarrow }\longrightarrow A,
\end{equation*}
of which we may think as the domain and codomain functors.

We are now referring to the example of \ref{ex7}, where we say what does it mean for an internal transformation to be natural with respect to another. Then we prove that an internal transformation is natural if and only if it is natural with respect to the \emph{arrow} transformation, and at the end we show that every internal natural transformation is a natural 2-cell.

\begin{proposition}\label{prop_naturality of transformations}
Let $\B$ be an arbitrary category and take $\C=\Cat(\B)$ with its 2-cell structure of internal transformations as in \ref{ex7}. A 2-cell $t \in H\left( A,B\right) $ is natural 
%with respect to a 2-cell $z\in H(X,A)$ if and only if 
%\begin{equation*}
%m\left\langle k_{1}z,tf_{0}\right\rangle =m\left\langle
%tg_{0},h_{1}z\right\rangle  \label{m<k1 z,t fo>=m<t go,h1 z>}.
%\end{equation*}
%Moreover, $t$ is natural 
if and only if it is natural with respect to the 2-cell
\begin{equation*}
( 1_{A_{1}},ed,ec) \in L\left(
A^{\rightarrow },A\right)\cong H\left(
A^{\rightarrow },A\right).
\end{equation*}
\end{proposition}

\begin{proof} Let us first observe what does it mean for $\mathbf{t}$ to be natural with respect to an arbitrary appropriate 2-cell $\mathbf{z}$. 
Suppose $\mathbf{t}=\left( t,h_1,k_1\right) \in L(A,B)\cong H\left( A,B\right) $ and $
\mathbf{z}=\left( z,f_1,g_1\right) \in L(X,A)\cong H\left( X,A\right)$ as illustrated,
\begin{equation*}
\xymatrix{
... & X_1
\ar@<.5ex>[r]^{}
\ar@<-.5ex>[r]^{}
\ar@<.5ex>[d]^{g_1}
\ar@<-.5ex>[d]_{f_1}
& X_0
\ar@<.5ex>[d]^{g_0}
\ar@<-.5ex>[d]_{f_0}
\ar[ld]_{z}
\ar[l]
\\
... & A_1
\ar@<.5ex>[r]^{}
\ar@<-.5ex>[r]^{}
\ar@<.5ex>[d]^{k_1}
\ar@<-.5ex>[d]_{h_1}
& A_0
\ar@<.5ex>[d]^{k_0}
\ar@<-.5ex>[d]_{h_0}
\ar[ld]_{t}
\ar[l]
\\
... & B_1
\ar@<.5ex>[r]^{}
\ar@<-.5ex>[r]^{}
& B_0
\ar[l]}
\end{equation*}%
then by the definition of a relation of horizontal composition (see Definition \ref{def nat2cellwrt}) we have that $\mathbf{t}\circ \mathbf{z}$ is equivalent to
\begin{equation}
 \left( k_{1}z,k_1f_1,k_1g_1\right)
+\left( tf_{0},h_1f_1,k_1f_1\right) =\left( tg_{0},h_1g_1,k_1g_1\right) +\left(
h_{1}z,h_1f_1,h_1g_1\right)  \notag 
\end{equation}
which simplifies to
\begin{equation}\notag
\left( m\left\langle k_{1}z,tf_{0}\right\rangle
,h_1f_1,k_1g_1\right) = \left( m\left\langle tg_{0},h_{1}z\right\rangle ,h_1f_1,k_1g_1\right)
\end{equation}%
which further simplifies to
\begin{equation}
m\left\langle k_{1}z,tf_{0}\right\rangle =m\left\langle
tg_{0},h_{1}z\right\rangle . \label{m<k1 z,t fo>=m<t go,h1 z>}
\end{equation}
Moreover, again by definition, $t$ is an internal natural transformation when%
\begin{equation}
m\left\langle k_{1},td\right\rangle =m\left\langle tc,h_{1}\right\rangle
\label{m<k1,td>=m<tc,h1>}
\end{equation}%
which is equivalent to saying that $\left( t,h_1,k_1\right) $ is natural
relative to $\left( 1_{A_{1}},ed,ec\right) ,$ as
displayed below%
\begin{equation*}
\xymatrix{
... & A_1
\ar@{=}[r]^{}
\ar@<.5ex>[d]^{ec}
\ar@<-.5ex>[d]_{ed}
& A_1
\ar@<.5ex>[d]^{c}
\ar@<-.5ex>[d]_{d}
\ar[ld]_{1}
\\
... & A_1
\ar@<.5ex>[r]^{}
\ar@<-.5ex>[r]^{}
& A_0
\ar[l]}.
\end{equation*}
\end{proof}

\begin{corollary}\label{corolary1}
Every internal natural transformation is a natural 2-cell.
\end{corollary}

\begin{proof}
Simply observe that
\begin{equation*}
\left( \ref{m<k1,td>=m<tc,h1>}\right) \Longrightarrow \left( \ref{m<k1 z,t
fo>=m<t go,h1 z>}\right)
\end{equation*}%
since%
\begin{eqnarray*}
m\left\langle k_{1},td\right\rangle z &=&m\left\langle tc,h_{1}\right\rangle
z \\
m\left\langle k_{1}z,tdz\right\rangle &=&m\left\langle
tcz,h_{1}z\right\rangle \\
m\left\langle k_{1}z,tf_{0}\right\rangle &=&m\left\langle
tg_{0},h_{1}z\right\rangle .
\end{eqnarray*}
\end{proof}

In this case there is a simple criteria which detects whether a 2-cell is natural or not without the need of comparing it with all the possible 2-cells. This seems to be an intrinsic phenomenon for this particular 2-cell structure in this particular type of category.

\section{Pseudocategories}\label{section 6}

The notion of pseudocategory (as considered in \cite{Ch1}) is only defined internally to a 2-category. Here we extend it to the more general context of a category with a 2-cell structure (or sesquicategory).

In any category $\mathbf{C}$, it is always possible to consider two trivial 2-cell structures (see Definition \ref{def 2cellstruct}), namely the discrete one, which is obtained when $H=\hom $
and $\dom ,\cod ,0,+$ are all identities, and the codiscrete one, which is obtained when $H=\hom \times \hom $, $\dom $ is the second projection, $%
\cod $ is the first projection, $0$ is the diagonal and $+$ is uniquely
determined. As already observed, a pseudocategory in the first case is an internal category to $%
\mathbf{C}$, while in the second case it is simply a precategory in $\mathbf{C}$.

When $\mathbf{C}=\Cat(\Set)$, and choosing for the 2-cell structure the natural transformations, then a pseudocategory is a pseudo-double-category in the sense of Grandis and Par\'e (see \cite{Grandis-Pare-01,Grandis-Pare-02}), which is at the same time a
generalization of a double-category and a bicategory.

At this level of generality, it appears that there is no particular reason why to prefer a specific 2-cell structure instead of another. For instance, Top, the category of topological spaces, is usually considered with the 2-cell structure which is obtained from the homotopy classes of homotopies, but there are perhaps others which could have been considered as well.

We are now going to extend the notion of pseudocategory onto its full generality.

Let $\mathbf{C}$ be an arbitrary category with
 $\left( H,\dom ,\cod ,0,+\right) $ a 2-cell structure defined over it.

Recall from Section \ref{section2} that an internal precategory in $\C$ is a system $$(C_0,C_1,C_2,d,c,e,m,\pi_1,\pi_2)$$ with $C_{0},C_{1},C_2$ objects and $d,c,e,m,\pi_1,\pi_2$ morphisms in $\C$, displayed as%
\begin{equation*}
\xymatrix{ C_2 \ar@<1.5ex>[r]^{\pi_2}\ar@<-1.5ex>[r]_{\pi_1}\ar[r]|{m} & C_1 \ar@<1ex>[r]^{d} \ar@<-1ex>[r]_{d} & C_0
\ar[l]|{e} },
\end{equation*}
such that the conditions (PC1), (PC2) and (PC3) are satisfied. We now need a reformulation of (PC3), denoted by (PC3*), to be used in the context of a 2-cell structure. We will say that a commutative square, $d\pi_1=c\pi_2$, such as the leftmost one in the diagram $(\ref{eq squares})$ below (compare with (PC3)) satisfies (PC3*)  if there are spans \[\xymatrix{C_2&C_3\ar[l]_{p_1}\ar[r]^{p_2}&C2},\quad \xymatrix{C_3&C_4\ar[l]_{p'_1}\ar[r]^{p'_2}&C3}\] such that all the squares
\begin{equation}\label{eq squares}
\quadrado
{ C_2  }       {  \pi_2 }     { C_1    }
{ \pi_1  }                  {   c  }
{ C_1  }       {  d }     {   C_0  }%
\ \ \ \ \ \ \
\quadrado
{ C_3  }       { p_2 }     { C_2    }
{ p_1  }                  {   \pi_1  }
{ C_2  }       {  \pi_2 }     {   C_1  }
\ \ \ \ \ \ \
\quadrado
{ C_4  }       { p'_2 }     { C_3    }
{ p'_1  }                  {   p_1  }
{ C_3  }       {  p_2 }     {   C_2  }
\end{equation}
are $H$-cartesian pullback squares (see Definition \ref{Def: cartesian 2-cell structure, Ch2}). In that case we have induced morphisms
\begin{eqnarray*}
e_{1} &=&\left\langle 1,ed\right\rangle :C_{1}\longrightarrow C_{2} \\
e_{2} &=&\left\langle ec,1\right\rangle :C_{1}\longrightarrow C_{2} \\
m_{1} &=&1\times m:C_{3}\longrightarrow C_{2} \\
m_{2} &=&m\times 1:C_{3}\longrightarrow C_{2} \\
i_{0} &=&\left\langle e_{1},e_{2}\right\rangle :C_{2}\longrightarrow C_{3}\\
m_3&=&1\times m\times 1\colon{C_4\to C_3}\\
m_4&=&m\times 1\times 1\colon{C_4\to C_3}\\
m_5&=&1\times 1\times m\colon{C_4\to C_3},
\end{eqnarray*}
and for appropriate 2-cells $\alpha \in H\left( C_{3},C_{1}\right) \ \text{and }\lambda ,\rho \in H\left(
C_{1},C_{1}\right)$ we have (see Definition \ref{def cartesian}) induced 2-cells \[(\alpha\times_{0_{1_{C_0}}} 0_{1_{C_1}}),(0_{1_{C_1}}\times_{0_{1_{C_0}}} \alpha)\in H(C_4,C_2)\] and \[(\rho\times_{0_{1_{C_0}}} 0_{1_{C_1}}),(0_{1_{C_1}}\times_{0_{1_{C_0}}} \lambda)\in H(C_2,C_2).\]
All the previous morphisms and 2-cells will be needed in the definition below.

\begin{definition}
\label{def pseudocat} A \emph{pseudocategory} internal to a category $\mathbf{C}$, with respect to a 2-cell structure $\left( H,\dom ,\cod ,0,+\right) $, is a system
\begin{equation*}
\left( C_{0},C_{1},C_2,d,c,e,m,\pi_1,\pi_2,\alpha ,\lambda ,\rho \right)
\end{equation*}%
in which $(C_0,C_1,C_2,d,c,e,m,\pi_1,\pi_2)$ is a precategory in $\C$, that is, $C_{0},C_{1},C_2$ are objects in $\C$, $d,c,e,m,\pi_1,\pi_2$ are morphisms in $\C$, which display as%
\begin{equation*}
\xymatrix{ C_2 \ar@<1.5ex>[r]^{\pi_2}\ar@<-1.5ex>[r]_{\pi_1}\ar[r]|{m} & C_1 \ar@<1ex>[r]^{d} \ar@<-1ex>[r]_{d} & C_0
\ar[l]|{e} },
\end{equation*}
satisfying the conditions (PC1), (PC2) and (PC3), where in the place of (PC3), the commutative square $d\pi_1=c\pi_2$ is required to satisfy (PC3*) (see $(\ref{eq squares})$ above). 
% \[\ de=1=ce,dm=d\pi _{2},cm=c\pi _{1}.\]
  Moreover $\alpha ,\lambda ,\rho $ are natural and invertible 2-cells
\begin{equation*}
\alpha \in H\left( C_{3},C_{1}\right) \ \text{and }\lambda ,\rho \in H\left(
C_{1},C_{1}\right)
\end{equation*}%
with%
\begin{eqnarray*}
\dom \left( \alpha \right) &=&mm_{1}\ ,\ \cod \left( \alpha
\right) =mm_{2} \\
\dom \left( \lambda \right) &=&me_{2}\ ,\ \dom \left( \rho \right)
=me_{1}\ ,\ \cod \left( \lambda \right) =1_{C_{1}}=\cod \left(
\rho \right)
\end{eqnarray*}%
and such that the following conditions hold true
\begin{eqnarray}
d\lambda =&0_{d}&=d\rho \nonumber\\
c\lambda =&0_{c}&=c\rho \nonumber\\
d\alpha =0_{d\pi _{2}p_{2}}\ &,&\ c\alpha =0_{c\pi _{1}p_{1}}\nonumber\\
\lambda e&=&\rho e\nonumber\\
\alpha m_4 +\alpha m_5  &=& m\left( \alpha \times 0_{1}\right) +\alpha m_3
+m\left( 0_{1}\times \alpha \right)   \label{axiom: pentagon coherence eq form} \\
m\left( 0_{1}\times
\lambda \right) &=& m\left( \rho \times 0_{1}\right) +\alpha i_{0} .  \label{axiom: midle triang coheren eq form}
\end{eqnarray}%
\end{definition}

The 2-cells $\alpha ,\lambda ,\rho $ are usually presented as
\begin{equation*}
\cela {C_3} {mm_1} {\alpha} {mm_2} {C_1}%
\ ,\
\cela {C_1} {me_2} {\lambda} {1} {C_1}%
,\
\cela {C_1} {me_1} {\rho} {1} {C_1}.
\end{equation*}%
Equations $\left( \ref{axiom: pentagon coherence eq form}\right) $ and $%
\left( \ref{axiom: midle triang coheren eq form}\right) $ correspond respectively to the
internal versions of the Pentagon and Middle Triangle in Mac Lane's Coherence Theorem \cite{ML},
presented diagrammatically as
\begin{equation}
\xymatrix@=2pc{ & \bullet  \ar[rr]^{ m  (0_{C_1}\times_{C_0}\alpha)}
             \ar[ldd] _{ \alpha  ( 1_{C_1} \times_{C_0} 1_{C_1} \times_{C_0} m ) }
 & & \bullet \ar[rdd]^{\alpha   ( 1_{C_1} \times_{C_0} m \times_{C_0} 1_{C_1} ) }  & \\
& & & & \\
\bullet   \ar[rrdd]_{\alpha  ( m \times_{C_0} 1_{C_1} \times_{C_0} 1_{C_1} ) }
  &  &  &  &  \bullet  \ar[ddll]^{m  (\alpha\times_{C_0} 0_{C_1})}  \\
& & & & \\
 & & \bullet &  &}
\label{pentagon coherence}
\end{equation}%
\begin{equation}
\xymatrix@=2pc{
 \bullet   \ar[rr]^{\alpha  i_0}
            \ar[rd]_{m  (0_{C_1} \times_{C_0} \lambda)}
&& \bullet   \ar[ld]^{m  (\rho \times_{C_0} 0_{C_1})} \\
& \bullet}
\label{triangle coherence}
\end{equation}%
and restated in terms of generalized elements as%
\begin{equation}
\xymatrix@=0.5pc{
& f(g(hk))
 \ar[rr]^{ f\alpha_{g,h,k} }
 \ar[ldd] _{ \alpha_{f,g,hk} }
 & & f((gh)k)
\ar[rdd]^{\alpha_{f,gh,k} }
& \\
& & & & \\
(fg)(hk)
 \ar[rrdd]_{\alpha_{fg,h,k} }
  &  &  &  &  (f(gh))k
\ar[ddll]^{\alpha_{f,g,h} k}  \\
& & & & \\
 & & ((fg)h)k &  &}
\tag{pentagon}
\end{equation}%
\begin{equation}
\xymatrix@=2pc{
f(1g)   \ar[rr]^{\alpha_{f,1,g}}
            \ar[rd]_{f \lambda_g}
&& (f1)g   \ar[ld]^{\rho_f g} \\
& fg}
\tag{midle triangle}
\end{equation}%
with $m\left\langle f,g\right\rangle =fg$.

We conclude this exposition with three results of application.

The first example is an instance of section \ref{Example: 2-cells in Cat},
the second example is an application of section \ref{Example 1}, while the
third one is from \ref{Remark20}.

\subsection{Pseudocategories internal to $\Cat(\B)$ with $\B$ (weakly) Ma'tsev}\label{Th WMC}

In the setting of section \ref{section2}, let $\mathbf{B}$ be
a (weakly) Mal'cev category (see for example \cite{Ch3}) and consider $\mathbf{C}=$Cat$\left( \mathbf{B}\right)$ equipped with the 2-cell structure  $$\mathbf{H}=\left( H,\dom ,\cod ,0,+\right) $$ as in section \ref{Example: 2-cells in Cat}.

As proved in \cite{Ch9}, the pair $(\C,\mathbf{H})$
is a (weakly) Mal'tsev sesquicategory in which every pullback square is $H$-cartesian. The following Theorem is also proved there.

\begin{theorem}\label{thm WMC} Let $\B$ be a (weakly) Mal'tsev category and consider $\C=\Cat(\B)$ with the 2-cell structure $\left( H,\dom ,\cod ,0,+\right)$ as in \ref{ex7}. A pseudocategory internal to $\mathbf{C}$ and relative to the given 2-cell structure, satisfying the additional condition%
\begin{equation}\label{additional condition}
\lambda e=0_{e}=\rho e
\end{equation}%
is completely determined by a reflexive graph in $\mathbf{C}$%
\begin{equation*}
\xymatrix{
C_1
\ar@<1ex>[r]^{d}
\ar@<-1ex>[r]_{c}
& C_0
\ar[l]|{e}}
\ \ ,\ \ de=1_{C_{0}}=ce
\end{equation*}%
together with 2-cells%
\begin{equation*}
\lambda ,\rho \in H\left( C_{1},C_{1}\right)
\end{equation*}%
satisfying the following conditions,%
\begin{equation*}
\cod \left( \lambda \right) =1_{C_{1}}=\cod \left( \rho \right)
\end{equation*}%
\begin{eqnarray*}
d\lambda  &=&0_{d}=d\rho  \\
c\lambda  &=&0_{c}=c\rho  \\
\lambda e &=&0_{e}=\rho e
\end{eqnarray*}%
and furthermore it is equipped with a morphism%
\begin{equation*}
m:C_{2}\longrightarrow C_{1}
\end{equation*}%
uniquely determined by%
\begin{equation*}
me_{1}=u\ \text{ , }\ me_{2}=v,
\end{equation*}
with $v=\dom \left( \lambda \right) $ and $u=\dom \left( \rho
\right)$, together with a 2-cell $\alpha \in H\left( C_{3},C_{1}\right) $, uniquely determined by
\begin{equation*}
\alpha i_{1}e_{1}=-\rho u\ ,\ \alpha i_{2}e_{1}=-u\lambda +\lambda u\ ,\
\alpha i_{2}e_{2}=\lambda v.
\end{equation*}%
\end{theorem}

\subsection{Pseudocategories in groups}

In this second example we describe pseudocategories internal to groups. It gives a \emph{relaxed}
notion of a crossed module $(X,B,\partial)$ in which we have the freedom for choosing an element $\delta $ in
the centre of \thinspace $X$.

A pseudocategory internal to the category of groups, with the 2-cell structure given as in \ref{ex8}, is
completely determined by a group homomorphism
\begin{equation*}
X\overset{\partial }{\longrightarrow }B,
\end{equation*}%
an action of $B$ in $X$ (denoted by $b\cdot x$) and a distinguished element
in $X$, say $\delta $, satisfying the following conditions
\begin{eqnarray*}
\partial \delta &=&0 \\
x &=&\delta +x-\delta \\
\partial \left( b\cdot x\right) &=&b\partial \left( x\right) b^{-1} \\
\partial \left( x\right) \cdot \bar{x} &=&x+\bar{x}-x.
\end{eqnarray*}%

In this case it is not difficult to describe the objects and the pseudomorphisms from the internal pseudocategory, with the respective isomorphisms $\alpha,\lambda$ and $\rho$. The objects are elements of $B$, the arrows are pairs $$\left(
x,b\right) :b\longrightarrow \partial x+b$$ and the composition of $$\left(
x^{\prime },\partial x+b\right) :\partial x+b\longrightarrow \partial
x^{\prime }+\partial x+b$$ with $$\left( x,b\right) :b\longrightarrow \partial
x+b$$ is $$\left( x^{\prime }+x-\delta +b\cdot \delta ,b\right)
:b\longrightarrow \partial x^{\prime }+\partial x+b.$$ The isomorphism
between 
$\left( 0,\partial x+b\right) \circ \left( x,b\right) =\left(
x,b\right) \circ \left( 0,b\right) $ and $\left( x,b\right) $
 is the element 
 $\left( \delta ,0\right) \in X\rtimes B$.
  Associativity is satisfied, since $$\left( x^{\prime \prime },\partial x^{\prime }+\partial x+b\right) \circ
\left( \left( x^{\prime },\partial x+b\right) \circ \left( x,b\right)
\right) =\left( \left( x^{\prime \prime },\partial x^{\prime }+\partial
x+b\right) \circ \left( x^{\prime },\partial x+b\right) \right) \circ \left(
x,b\right) .$$

\subsection{The additive case}\label{the abelian case}
In the case when $\mathbf{A}$ is an additive category with kernels, equipped with  a 2-cell
structure which is given by an Ab-functor%
\begin{equation*}
H:\mathbf{A}^{op}\times \mathbf{A\longrightarrow }\Ab
\end{equation*}%
and a natural transformation%
\begin{equation*}
D:H\longrightarrow \hom _{\mathbf{A}},
\end{equation*}%
as in \ref{ex12}, \ref{ex13}, \ref{ex14},
then a pseudocategory internal to it is completely determined by
\begin{equation*}
A\overset{h}{\longrightarrow }B
\end{equation*}%
\begin{eqnarray*}
\lambda ,\rho &\in &H\left( A,A\right) \\
\eta &\in &H\left( B,A\right)
\end{eqnarray*}%
\begin{eqnarray*}
h\lambda &=&0 \\
h\rho &=&0 \\
h\eta &=&0
\end{eqnarray*}
with $\alpha $ uniquely determined.
% be equations $\left( \ref{eq_a2}\right) ,\left( \ref{eq_a3}%
%\right) $ and $\left( \ref{eq_a4}\right)$ presented in the next section.
The pseudocategory thus determined is of the form (see \cite{Ch5})%
\begin{equation*}
A\oplus A\oplus B\overset{m}{\longrightarrow }A\oplus B\underset{\left(
h \;1\right) }{\overset{\left( 0\;1\right) }{\underrightarrow{%
\overrightarrow{\overset{\binom{0}{1}}{\longleftarrow }}}}}B
\end{equation*}%
\begin{equation*}
m=\left(
\begin{array}{ccc}
f & g & h \\
0 & 0 & 1%
\end{array}%
\right)
\end{equation*}%
\begin{eqnarray*}
g &=&1-D\left( \lambda \right) \\
f &=&1-D\left( \rho \right) \\
h &=&-D\left( \eta \right)
\end{eqnarray*}%
\begin{equation*}
\alpha =\left(
\begin{array}{cccc}
\alpha _{1} & \alpha _{2} & \alpha _{3} & \alpha _{0} \\
0 & 0 & 0 & 0%
\end{array}%
\right)
\end{equation*}%
\begin{eqnarray*}
\alpha _{1} &=&-f\rho \\
\alpha _{2} &=&\lambda +g\rho -\rho -f\lambda \\
\alpha _{3} &=&g\lambda -f\eta \delta \\
\alpha _{0} &=&g\eta -f\eta
\end{eqnarray*}%
\begin{equation*}
\lambda =\left(
\begin{array}{cc}
\lambda & \eta \\
0 & 0%
\end{array}%
\right) \ \ \ \ \ \ \ \ \ \ \ \rho =\left(
\begin{array}{cc}
\rho & \eta \\
0 & 0%
\end{array}%
\right),
\end{equation*}
with obvious abuse of notation for $\lambda$ and $\rho$.

In particular the category of abelian chain complexes is of this form but in which we have to specify that $\lambda$ and $\rho$ are natural 2-cells in the sense of Definition \ref{def natcell}. See also the last remark in the conclusion of this article.

Another example of this form is the category TAG, of Topological Abelian Groups (see example \ref{ex15}). A pseudocategory in TAG (with the
2-cell structure as in \ref{ex15}) is completely determined by a morphism in TAG%
\begin{equation*}
k:A\longrightarrow B,
\end{equation*}%
together with
\begin{equation*}
\lambda ,\rho :I\times A\longrightarrow A
\end{equation*}%
in $H\left( A,A\right) $ and also
\begin{equation*}
\eta :I\times B\longrightarrow A
\end{equation*}%
in $H\left( B,A\right) $ satisfying%
\begin{equation*}
k\rho \left( t,\_\right) =0,k\lambda \left( t,\_\right) =0,k\eta \left(
t,\_\right) =0.
\end{equation*}%
The objects in the pseudocategory are the points in $B$ while the
pseudomorphisms are pairs $\left( a,b\right) $ with domain $b$ and codomain $%
k\left( a\right) +b$; the composition of
\begin{equation*}
\xymatrix{b \ar[r]^{(a,b)} & b' \ar[r]^{(a',b')} & b''}
\end{equation*}
is given by the following formula%
\begin{equation*}
\left( a-\rho \left( 1,a\right) +a^{\prime }-\lambda \left( 1,a^{\prime
}\right) -\eta \left( 1,b\right) ,b\right) .
\end{equation*}%
In particular, if we consider that $A$ is the space of paths in $B$ starting
at zero%
\begin{equation*}
A=\left\{ x:I\longrightarrow B\mid x\text{ continuous and }x\left( 0\right)
=0\right\}
\end{equation*}%
with
\begin{equation*}
k\left( x\right) =x\left( 1\right)
\end{equation*}%
and choose representatives of $\lambda ,\rho $ and $\eta $ as%
\begin{eqnarray*}
\lambda \left( s,x\right) \left( t\right) &=&\left\{
\begin{tabular}{ccc}
$x\left( st\right) -x\left( 2st\right) $ & if & $t\leqslant \frac{1}{2}$ \vspace{.2cm}\\
$x\left( st\right) -x\left( s\right) $ & if & $t\geqslant \frac{1}{2}$\vspace{.2cm}
\end{tabular}%
\right. \\
\rho \left( s,x\right) \left( t\right) &=&\left\{
\begin{tabular}{ccc}
$x\left( st\right) $ & if & $t\leqslant \frac{1}{2}$ \vspace{.2cm}\\
$x\left( st\right) -x\left( 2st-s\right) $ & if & $t\geqslant \frac{1}{2}$\vspace{.2cm}
\end{tabular}%
\right. \\
\eta &=&0
\end{eqnarray*}%
then we obtain the usual composition of paths%
\begin{equation*}
y+x=\left\{
\begin{tabular}{ccc}
$x\left( 2t\right) $ & if & $t\leqslant \frac{1}{2}$ \vspace{.2cm}\\
$y\left( 2t-1\right) +x\left( 1\right) $ & if & $t\geqslant \frac{1}{2}$%
\end{tabular}%
\right. .
\end{equation*}

\section{Conclusion}

As it follows from Definition \ref{def natcell}, recognizing whether a given 2-cell $x\in H\left( A,B\right) $ is a natural one is generally a complicated task: we have to analyse equation $\left( \ref{eq_a}%
\right) $ for every possible $y$. On the other hand, if removing naturality
conditions for $\alpha ,\lambda ,\rho $, we loose the Coherence Theorem \cite%
{ML} and we have no longer guaranteed that for example the 
diagrams 
\begin{equation}
\xymatrix@=2pc{
1(fg)   \ar[rr]^{\alpha_{1,f,g}}
            \ar[rd]_{\lambda_{fg} }
&& (1f)g   \ar[ld]^{\lambda_f g} \\
& fg}
\label{left triangle}
\end{equation}%
\begin{equation}
\xymatrix@=2pc{
f(g1)   \ar[rr]^{\alpha_{f,g,1}}
            \ar[rd]_{f \rho_g}
&& (fg)1   \ar[ld]^{\rho_{fg} } \\
& fg}  \label{right triangle}
\end{equation}
are commutative. This diagrams when internalized correspond respectively to the following
equations%
\begin{eqnarray*}
m\left( \lambda \times 0_{C_{1}}\right) +\alpha i_{2} &=&\lambda m, \\
\rho m+\alpha i_{1} &=&m\left( 0_{C_{1}}\times \rho \right)
\end{eqnarray*}%
and since the 2-cells are assumed to be invertible they can be presented as%
\begin{eqnarray*}
\alpha i_{2} &=&-m\left( \lambda \times 0_{C_{1}}\right) +\lambda m, \\
\alpha i_{1} &=&-\rho m+m\left( 0_{C_{1}}\times \rho \right) .
\end{eqnarray*}

The examples illustrated in the previous section suggest an intermediate notion for an \emph{unnatural} pseudocategory, where we do not ask for the 2-cells $\alpha,\lambda,\rho$ to be natural, but only that they are natural with respect to each other. In fact in all those cases the 2-cell $\alpha$ is completely determined and hence (at least for a further study of those cases) we may ask that only $\lambda$ and $\rho$ are natural with respect to each other, which in other words means that the horizontal compositions 
\begin{equation}
\lambda \circ \lambda ,\lambda \circ \rho ,\rho \circ \rho ,\rho \circ
\lambda  \label{eq_a5}
\end{equation}
are defined, or even that their commutators vanish (Definition \ref{def commutator}). In order to have some control on the coherence aspects we also ask that at least the conditions
\begin{eqnarray}
%m\left( \alpha \times 0_{1}\right) +\alpha \left( 1\times m\times 1\right)
%+m\left( 0_{1}\times \alpha \right) &=&\alpha \left( m\times 1\times
%1\right) +\alpha \left( 1\times 1\times m\right)  \notag \\
%&&  \label{eq_a1} \\
%\alpha i_{0} &=&-m\left( \rho \times 0_{1}\right) +m\left( 0_{1}\times
%\lambda \right)  \label{eq_a2} \\
\alpha i_{2} &=&-m\left( \lambda \times 0_{C_{1}}\right) +\lambda m
\label{eq_a3} \\
\alpha i_{1} &=&-\rho m+m\left( 0_{C_{1}}\times \rho \right) ,  \label{eq_a4}
\end{eqnarray}%
are satisfied.
%
%For that we consider an arbitrary category $\C$ and suppose it has a 2-cell
%structure $\left( H,\dom ,\cod ,0,+\right) $ defined over it.
%
%A (non natural) pseudocategory internal to $\mathbf{C}$ and relative to the
%2-cell structure $\left( H,\dom ,\cod ,0,+\right) $ is a system%
%\begin{equation*}
%\left( C_{0},C_{1},d,c,e,m,\alpha ,\lambda ,\rho \right)
%\end{equation*}%
%satisfying
%
%where $C_{0},C_{1}$ are objects, $d,c,e,m$ are morphisms as in Definition \ref{def pseudocat}, and, $$\alpha ,\lambda ,\rho $$ are invertible 2-cells $%
%\alpha \in H\left( C_{3},C_{1}\right) \ $and $\lambda ,\rho \in H\left(
%C_{1},C_{1}\right) $ satisfying the following conditions (with $%
%C_{2},C_{3},\pi _{1},\pi _{2},p_{1},p_{2},e_{1},e_{2},m_{1},m_{2},i_{0}$
%defined as in Definition \ref{def pseudocat})%
%\begin{eqnarray*}
%\dom \left( \alpha \right) &=&mm_{1}\ ,\ \cod \left( \alpha
%\right) =mm_{2} \\
%\dom \left( \lambda \right) &=&me_{2}\ ,\ \dom \left( \rho \right)
%=me_{1}\ ,\ \cod \left( \lambda \right) =1_{C_{1}}=\cod \left(
%\rho \right)
%\end{eqnarray*}%
%\begin{eqnarray*}
%d\lambda &=&0_{d}=d\rho \\
%c\lambda &=&0_{c}=c\rho \\
%d\alpha &=&0_{d\pi _{2}p_{2}}\ ,\ c\alpha =0_{c\pi _{1}p_{1}}
%\end{eqnarray*}%
%\begin{equation*}
%\lambda e=\rho e
%\end{equation*}%
%

Clearly when $\alpha ,\lambda ,\rho $ are natural 2-cells then
these three conditions are redundant and we obtain Definition \ref{def
pseudocat}.

For instance, in example \ref{Th WMC}, as also proved in \cite{Ch9},  we have that
\begin{equation*}
(\ref{eq_a3})+(\ref{eq_a4})+(\ref{eq_a5})\Rightarrow (\ref{eq_a1})+(\ref{eq_a2}),
\end{equation*}%
with $(\ref{eq_a1}),(\ref{eq_a2})$ the \emph{pentagon} and the \emph{midle triangle} identities restated as: \begin{eqnarray}
m\left( \alpha \times 0_{1}\right) +\alpha m_3
+m\left( 0_{1}\times \alpha \right) &=&\alpha m_4 +\alpha m_5  \label{eq_a1} \\
\alpha i_{0} &=&-m\left( \rho \times 0_{1}\right) +m\left( 0_{1}\times
\lambda \right).  \label{eq_a2}
\end{eqnarray}%

In the example of section \ref{the abelian case}
 and using the notation
\begin{equation*}
\left[ x,y\right] =D\left( x\right) y-xD\left( y\right)
\end{equation*}%
we have that $\alpha$ is determined by $(\ref{eq_a3})+(\ref{eq_a4})$ but  $\left( \ref{eq_a1}\right) $ is no longer a trivial condition, it becomes equivalent to
\begin{equation*}
\left( 1-D\rho \right) \left[ \rho ,\rho \right] =0
\end{equation*}%
\begin{equation*}
D\lambda \left[ \rho ,\rho \right] =D\rho \left[ \lambda ,\rho \right]
+\left( 1-D\rho \right) \left[ \rho ,\lambda \right]
\end{equation*}%
\begin{equation*}
\left( 1-D\lambda \right) \left[ \lambda ,\rho \right] +D\lambda \left[ \rho
,\lambda \right] =D\rho \left[ \lambda ,\lambda \right] +\left( 1-D\rho
\right) \left[ \rho ,\eta \right] h
\end{equation*}%
\begin{equation*}
\left( 1-D\lambda \right) \left[ \lambda ,\lambda \right] =\left( 1-D\lambda
\right) \left[ \lambda ,\eta \right]
\end{equation*}%
\begin{equation*}
\left( 1-D\rho -D\lambda \right) \left[ \lambda ,\eta \right] =\left(
1-D\rho -D\lambda \right) \left[ \rho ,\eta \right]
\end{equation*}%
which, however, is  trivial as soon as we introduce $\left( \ref{eq_a5}\right) $.

\end{document}